\documentclass{amsart}

\usepackage[title]{appendix}
\usepackage{amssymb}
\usepackage{amsmath}
\usepackage{amsthm}
\usepackage{tikz-cd}
\usepackage{mathtools}
\usepackage{framed}
\usepackage{comment}

\usepackage[margin=30mm]{geometry}

\newcommand{\Q}{\mathbb{Q}}
\newcommand{\R}{\mathbb{R}}

\newcommand{\Z}{\mathbb{Z}}
\renewcommand{\H}{\mathbb{H}}
\newcommand{\F}{\mathbb{F}}

\newcommand{\Alg}{\text{\normalfont{Alg}}}

\newcommand{\Proj}{\text{\normalfont{Proj}}}

\newcommand{\hProj}{\text{\normalfont{hProj}}}
\newcommand{\Mod}{\text{\normalfont Mod}}

\DeclareMathOperator{\Aut}{Aut}

\DeclareMathOperator{\Hom}{Hom}
\DeclareMathOperator{\id}{id}

\DeclareMathOperator{\rank}{rank}

\DeclareMathOperator{\Ext}{Ext}

\DeclareMathOperator{\IM}{Im}

\DeclareMathOperator{\Ker}{Ker}

\DeclareMathOperator{\GL}{GL}

\DeclareMathOperator{\Um}{Um}
\DeclareMathOperator{\D}{D2}
\DeclareMathOperator{\SF}{SF}
\DeclareMathOperator{\PHT}{PHT}
\DeclareMathOperator{\HT}{HT}
\DeclareMathOperator{\FP}{FP}
\DeclareMathOperator{\FF}{F}

\newtheorem{thm}{Theorem}[section]

\newtheorem*{thm*}{Theorem}
\newtheorem{prop}[thm]{Proposition}
\newtheorem*{prop*}{Proposition}
\newtheorem{lemma}[thm]{Lemma}
\newtheorem{corollary}[thm]{Corollary}
\newtheorem*{corollary*}{Corollary}

\newtheorem{question}[thm]{Question}
\newtheorem*{question*}{Question}

\theoremstyle{definition}

\newtheorem*{theorem*}{Theorem}

\theoremstyle{remark}
\newtheorem{remark}[thm]{Remark}
\newtheorem*{remark*}{Remark}

\usepackage{rotating}

\usepackage{color}
\usepackage{enumerate}

\usepackage{cite}
\usepackage[nodisplayskipstretch]{setspace}
\usepackage{placeins}

\newenvironment{clist}[1]
{\begin{enumerate}[\normalfont #1]}
{\end{enumerate}}

\newcommand{\wh}{\widehat}
\newcommand{\wt}{\widetilde}

\newcommand{\xrightarrowdbl}[2][]{%
  \xrightarrow[#1]{#2}\mathrel{\mkern-14mu}\rightarrow
}

\makeatletter
\@namedef{subjclassname@2020}{%
  \textup{2020} Mathematics Subject Classification}
\makeatother

\begin{document}

\title[Projective modules and homotopy classification]{Projective modules and the homotopy classification of $(G,n)$-complexes}

\author{John Nicholson}
\address{Department of Mathematics, Imperial College London, London, SW7 2AZ, U.K.}

\email{john.nicholson@imperial.ac.uk}



\begin{abstract}
A $(G,n)$-complex is an $n$-dimensional CW-complex with fundamental group $G$ and whose universal cover is $(n-1)$-connected.
If $G$ has periodic cohomology then, for appropriate $n$, we show that there is a one-to-one correspondence between the homotopy types of finite $(G,n)$-complexes and the orbits of the stable class of a certain projective $\Z G$-module under the action of $\Aut(G)$. 
We develop techniques to compute this action explicitly and use this to give an example where the action is non-trivial.
\end{abstract}

\maketitle

\section{Introduction}

For a group $G$ and $n \ge 2$, a \textit{$(G,n)$-complex} is a connected $n$-dimensional CW-complex $X$ for which $\pi_1(X) \cong G$ and $\widetilde{X}$ is $(n-1)$-connected. Equivalently, it is the $n$-skeleton of a $K(G,1)$.
For example, a finite $(G,2)$-complex is equivalently a finite $2$-complex $X$ with $\pi_1(X) \cong G$. An example of a finite $(G,3)$-complex is a closed 3-manifold $M$ with $\pi_1(M) \cong G$ finite.
Given a group $G$ and $n \ge 2$, a finite $(G,n)$-complex exists if and only if $G$
has type $\FF_n$ in the sense of Wall \cite{Wa65}.

Let $\HT(G,n)$ be the set of homotopy types of finite $(G,n)$-complexes, which can be viewed as a graph with edges between each $X$ and $X \vee S^n$. It is well known that $\HT(G,n)$ is a tree \cite{Wh39}, i.e. a connected acyclic graph, and has a grading coming from $(-1)^n \chi(X)$ which takes a minimum value $\chi_{\min}(G,n)$. 
The problem of determining the structure of $\HT(G,n)$ as a tree has a long history which dates back to Cockcroft-Swan \cite{CS61} and Dyer-Sieradski \cite{DS73,DS75}. 

In the case of finite abelian groups, the structure of $\HT(G,n)$ has been classified through a series of articles by Metzler \cite{Me76}, Dyer-Sieradski \cite{SD79}, Browning \cite{Br79b} and Linnell \cite{Li93}. 
However, much less is known for non-abelian groups and an important class of examples are the groups with $k$-periodic cohomology, i.e. finite groups for which the Tate cohomology groups satisfy $\widehat{H}^{i}(G;\Z) \cong \widehat{H}^{i+k}(G;\Z)$ for all $i \in \Z$. 
For example, if $G$ is finite and $n$ is even, then it was shown by Browning \cite{Br79a} that $\chi(X) = \chi(Y)$ implies $X \vee S^n \simeq Y \vee S^n$ (see also \cite{HK93}). However, when $n$ is odd, this is known only when $G$ does not have $k$-periodic cohomology for  $k \mid n+1$ (see Question \ref{question:fork}).

The aim of this article is to make new progress towards the classification over groups with periodic cohomology, building upon work of Dyer \cite{Dy76} and Johnson \cite{Jo03a}.

\subsection{Main results}

Let $\PHT(G,n)$ be the tree of polarised homotopy types of finite $(G,n)$-complexes, i.e. the homotopy types of pairs $(X, \rho)$ where $\rho : \pi_1(X) \cong G$.

Let $G$ be a finite group and let $C(\Z G)$ denote the projective class group, i.e. the equivalence classes of finitely generated projective $\Z G$-modules where $P \sim Q$ if $P \oplus \Z G^i \cong Q \oplus \Z G^j$ for some $i, j$.
Note that a class $[P] \in C(\Z G)$ can be viewed as the set of (non-zero) projective $\Z G$-modules $P_0$ for which $P_0 \sim P$, and this has the structure of a graded tree with edges between each $P_0$ and $P_0 \oplus \Z G$. 
Let $T_G \le C(\Z G)$ denote the Swan subgroup (see Section \ref{ss:swan-modules}).
If $G$ has $k$-periodic cohomology, then the \textit{Swan finiteness obstruction} is an element $\sigma_k(G) \in C(\Z G)/T_G$ which vanishes if and only if there exists a finite CW-complex $X$ with $\pi_1(X) \cong G$ and $\wt X \simeq S^{k-1}$.

Recall that a finitely presented group $G$ has the \textit{D2 property} if every cohomologically 2-dimensional finite complex $X$ with $\pi_1(X) \cong G$ is homotopy equivalent to a finite 2-complex.

\begingroup
\renewcommand\thethm{\Alph{thm}}
\begin{thm} \label{thm:main-topological-I}
Let $G$ have $k$-periodic cohomology and let $n=ik$ or $ik-2$ for some $i \ge 1$. Then there is an injective map of graded trees
\[ \Psi: \PHT(G,n) \to [P_{(G.n)}]\] 
for any projective $\Z G$-module $P_{(G.n)}$ with $\sigma_{ik}(G) = [P_{(G.n)}] \in C(\Z G)/T_G$.
Furthermore, $\Psi$ is bijective if and only if $n \ge 3$ or if $n=2$ and $G$ has the {\normalfont D2} property.
\end{thm}
\endgroup

\setcounter{thm}{\value{thm}-1}
\begin{remark}
(a) If $G$ satisfies the Eichler condition, then $[P_{(G,n)}]$ has \textit{cancellation} in the sense that $P_1 \oplus \Z G \cong P_2 \oplus \Z G$ implies $P_1 \cong P_2$ for all $P_1$, $P_2 \in [P_{(G,n)}]$ (see \cite{Ja68}). This implies that $\PHT(G,n)$ and $\HT(G,n)$ have cancellation in the sense that $X \vee S^n \simeq Y \vee S^n$ implies $X \simeq Y$, and recovers the main result of Dyer \cite{Dy76}.

\noindent (b) An equivalent statement appeared in \cite{Jo03a} in the case $n=2$, though the proof contained a small gap which was patched up in \cite{Ni19} using a theorem of Browning \cite{Br79a}.
\end{remark}

Our proof is based on the work of Hambleton-Kreck \cite{HK93} and is independent of \cite{Br79a, Jo03a}.
After establishing preliminaries in Sections \ref{section:ext} and \ref{section:swan-modules-and-projective-extensions}, we will prove general cancellation theorems for chain complexes of projective modules in Section \ref{section:cancellation-of-projective-extensions}. This suffices to prove Theorem \ref{thm:main-topological-I} due to the correspondence between $\PHT(G,n)$ and the tree of algebraic $n$-complexes (see Proposition \ref{prop:pht-to-alg}).
In Theorem \ref{thm:main-topological-I-detailed}, we give a detailed version of Theorem \ref{thm:main-topological-I} which contains an explicit description of the map $\Psi$. 

We then use of this description of $\Psi$ to determine the induced action of $\Aut(G)$ on $[P_{(G,n)}]$ via the bijection $\HT(G,n) \cong \PHT(G,n)/\Aut(G)$.
To state the induced action, consider the following two operations for $M$ a (left) projective $\Z G$-module: 
\begin{enumerate}
\item
If $\theta \in \Aut(G)$, then let $M_\theta$ be the $\Z G$-module whose abelian group is that of $M$ but with action $g \cdot x = \theta(g)x$ for $g \in G$, $x \in M$ (see Lemma \ref{lemma:theta-basic-facts}).
\item
If $r$ represents a class in $(\Z/|G|)^\times$ and $I \subseteq \Z G$ is the augmentation ideal, then $(I,r)$ is a projective $\Z G$-module. The tensor product $(I,r) \otimes M$ is a projective $\Z G$-module since $(I,r)$ is a two-sided ideal (see Lemma \ref{lemma:tensoring-with-swan-modules}).
\end{enumerate}

In Section \ref{section:homotopy-classification}, we will prove the following which is the main result of this article. Note that every projective $\Z G$-module has the form $P \oplus \Z G^r$ where $P$ has rank one and $r \ge 0$ (see Section \ref{ss:ZG-mod}).

\begingroup
\renewcommand\thethm{\Alph{thm}}
\begin{thm} \label{thm:main-topological-II}
Let $G$ have $k$-periodic cohomology and let $n=ik$ or $ik-2$ for some $i \ge 1$. Then $\Psi$ induces an injective map of graded trees
\[ \bar{\Psi}: \HT(G,n) \to [P_{(G.n)}]/\Aut(G),\]
where the action by $\theta \in \Aut(G)$ is given by 
\[ \theta : P \oplus \Z G^r  \mapsto ((I, \psi_k(\theta)^i) \otimes P_\theta) \oplus \Z G^r\] 
where $P$ has rank one, for some map $\psi_k : \Aut(G) \to (\Z/|G|)^\times$ which depends only on $G$ and $k$.
Furthermore, $\bar{\Psi}$ is bijective if and only if $n \ge 3$ or if $n=2$ and $G$ has the {\normalfont D2} property.
\end{thm}
\endgroup

This reduces the problem of determining when cancellation occurs in the homotopy trees to the purely algebraic problem of determining cancellation for $[P]$ and $[P]/\Aut(G)$ which will be dealt with in  \cite{Ni20b}.

\subsection{Computing the action of $\Aut(G)$}

After proving Theorems \ref{thm:main-topological-I} and \ref{thm:main-topological-II}, the remainder of this article will be devoted to exploring the action of $\Aut(G)$ on $[P_{(G,n)}]$. This includes establishing some general theory in preparation for the more detailed computations in \cite{Ni20b}.

Firstly, and perhaps somewhat surprisingly, we could find no example where the $\Aut(G)$-action described in Theorem \ref{thm:main-topological-II} does not take the form of the simpler action $P \mapsto P_\theta$. In all examples computed, we had $(I, \psi_k(\theta)) \cong \Z G$ which implies that $(I,\psi_k(\theta)^i) \cong \Z G$. If $P \oplus \Z G^r \in [P_{(G,n)}]$ where $P$ has rank one, then this implies that $\theta(P) \cong P_{\theta} \oplus \Z G^r \cong (P \oplus \Z G^r)_{\theta}$. In particular, we have that $\theta(P) \cong P_\theta$ for all $P \in [P_{(G,n)}]$.
We therefore ask the following:

\begingroup
\renewcommand\thethm{\ref{question:is-swan-free}}
\begin{question} 
Does there exist $G$ with $k$-periodic cohomology and $\theta \in \Aut(G)$ for which $(I,\psi_k(\theta))$ is not free?	
\end{question}
\endgroup

There are two approaches to finding examples where $(I,\psi_k(\theta))$ is not free. The first is to find an example where $(I,\psi_k(\theta))$ is not even stably free. It was shown by Dyer \cite[p276]{Dy76} and Davis \cite{Da83} that $(I,\psi_k(\theta))$ is stably free when $\sigma_k(G)=0$. Davis asked whether this also holds when $\sigma_k(G) \ne 0$ \cite[p488]{Da83}. The second approach is to find an example where $(I,\psi_k(\theta))$ is stably free but not free. This is likely to be difficult since the general question of whether $(I,r)$ can be stably free but not free is still open and dates back to Wall's 1979 Problems List \cite[Problem A5]{Wa79b}.

In Section \ref{section:Milnor-squares}, we develop a general method to compute the action $P \mapsto P_\theta$. We will then use this to give the following example where the action is non-trivial.
Let $Q_{4n}$ denote the quaternion group of order $4n$, which has $4$-periodic cohomology. Since $\sigma_4(Q_{4n})=0$, we can take $[P_{(Q_{4n},2)}] = [\Z Q_{4n}] = \bigcup_{r \ge 1} \SF_r(\Z Q_{4n})$ where $\SF_r(\Z Q_{4n})$ is the set of stably free $\Z Q_{4n}$-modules of rank $r \ge 1$.
As above, let $\theta \in \Aut(Q_{4n})$ act on $[\Z Q_{4n}]$ by $\theta : P \mapsto (I, \psi_4(\theta)^i) \otimes P_\theta$ for some $i \ge 1$. We show:

\begingroup
\renewcommand\thethm{\ref{thm:main-example}}
\begin{thm}
$\Aut(Q_{24})$ acts non-trivially on $[\Z Q_{24}]$. More specifically, we have $|\SF_1(\Z Q_{24})| = 3$ and $|\SF_1(\Z Q_{24})/\Aut(Q_{24})| = 2$. 
\end{thm}
\endgroup

This is in contrast to the case $Q_{4n}$ for $2 \le n \le 5$ where  $|\SF_1(\Z Q_{4n})| = 1$ and the case $Q_{28}$ where $|\SF_1(\Z Q_{28})| = |\SF_1(\Z Q_{28})/\Aut(Q_{28})| = 2$ (see Figure \ref{figure:quaternion-table}).

\subsection{Overview of the wider project}

This article is the first of a two-part series (followed by \cite{Ni20b}) in which we explore the classification of finite $(G,n)$-complexes over groups with periodic cohomology. 
These results are motivated by the following.

\subsubsection*{Wall's D2 problem for groups with $4$-periodic cohomology}

In the language above, the D2 problem asks whether every finitely presented group $G$ has the D2 property. This dates back to Wall's 1965 paper on finiteness conditions \cite{Wa65} and is currently open.
The case where $G$ has 4-periodic cohomology was proposed to contain a counterexample to the D2 problem in 1977 \cite{Co77}, and has since been studied extensively.
In this case, Johnson proved Theorem \ref{thm:main-topological-I} when $n=2$ and, using results of Swan \cite{Sw83}, he established the D2 property for many new groups \cite{Jo03a}. 
In \cite{Ni18,Ni19}, we extended these results and determined when $\PHT(G,2)$ has cancellation.

In the case where $\PHT(G,2)$ has non-cancellation, the D2 property has only been proven for $Q_{28}$ (see \cite{MP19,Ni19}). This motivated Theorem \ref{thm:main-topological-II} in the case $n=2$ since one imagines it might be easier to prove that $\bar{\Psi}$ is bijective rather than $\Psi$. The question of when $\HT(G,2)$ has cancellation is answered in \cite[Theorem A]{Ni20b}.

\subsubsection*{Stable and unstable classification of manifolds}

If $X$ is a finite $(G,n)$-complex, then there exists an embedding $i : X \hookrightarrow \R^{2n+1}$. The boundary of a smooth regular neighbourhood of $i$ gives a smooth closed $2n$-manifold $M(X)$. If $X$ is determined up to simple homotopy, then $M(X)$ is well-defined up to $s$-cobordism which coincides with homeomorphism in the case where $G$ is finite by work of Freedman. Furthermore, we have that $M(X \vee S^n) \cong M(X) \# (S^n \times S^n)$.
This can be found in \cite[Section 5]{BC+21}.

In 1984, Kreck-Schafer used this to construct smooth closed $4n$-manifolds $M_1$ and $M_2$ for every $n \ge 1$ such that $M_1 \# (S^{2n} \times S^{2n}) \cong M_2 \# (S^{2n} \times S^{2n})$ are diffeomorphic but $M_1 \not \simeq M_2$ \cite{KS84}. Their examples have the form $M(X_i)$ where the $X_i \in \HT(G,n)$ are the non-cancellation examples for $G$ abelian found by Metzler, Dyer and Sieradski \cite{Me76, Si77, SD79}.
Recently, Conway-Crowley-Powell-Sixt constructed examples of both simply-connected $M_i$ \cite{CCPS21a} and infinitely many $M_i$ \cite{CCPS21b} for all $n \ge 2$. However, the examples of Kreck-Schafer remain the only known examples in dimension 4.
In classifying $\HT(G,n)$ when $G$ has periodic cohomology, we hope to create a second family of examples both in dimension 4 and in higher dimensions. 

\subsection*{Conventions}

All rings $R$ will be assumed to have a multiplicative identity and all $R$-modules will be assumed to be finitely generated left $R$-modules.

Recall that groups with periodic cohomology are necessarily finite. For most of this article, we will therefore restrict to the case where $G$ is a finite group. However, we will briefly consider finitely presented groups more generally at the start of Sections \ref{section:classification-of-Alg(G,n)} and \ref{section:homotopy-classification}.

\subsection*{Acknowledgements}
I would like to thank F. E. A. Johnson for his guidance and many interesting conversations.
I would also like to thank Ian Hambleton, Jonathan Hillman and Wajid Mannan for helpful comments and references. 
Finally, I would like to thank the anonymous referee for their careful reading and suggestions which improved the exposition of this article.
This work was supported by the UK Engineering and Physical Sciences Research Council (EPSRC) grant EP/N509577/1 as well as the Heilbronn Institute for Mathematical Research.

\section{Extensions of modules} \label{section:ext}

Let $R$ be a ring. Recall our convention that all $R$-modules are assumed to be finitely generated left $R$-modules.
 For $R$-modules $A$ and $B$, define $\Ext_{R}^{n}(A,B)$ to be the set of exact sequences
\[ E = ( 0 \to B \xrightarrow[]{i} E_{n-1} \xrightarrow[]{\partial_{n-1}} E_{n-2} \xrightarrow[]{\partial_{n-2}} \cdots \xrightarrow[]{\partial_{2}} E_1 \xrightarrow[]{\partial_{1}} E_0 \xrightarrow[]{\varepsilon} A \to 0)\]
for $R$-modules $E_i$ considered up to congruence, i.e. the equivalence relation generated by \textit{elementary congruences} which are chain maps of the form
\[
\begin{tikzcd}[row sep=.5cm, column sep=small]
E \ar[d,"\varphi"] \\
E'
\end{tikzcd}
= \left( 
\begin{tikzcd}[row sep=.5cm, column sep=small]
0 \ar[r] & B \ar[r] \ar[d,"\id"] & E_{n-1} \ar[r] \ar[d,"\varphi_{n-1}"] & \cdots \ar[r] & E_0 \ar[r] \ar[d,"\varphi_0"] & A \ar[r] 
\ar[d,"\id"] & 0  \\
0 \ar[r] & B \ar[r] & E_{n-1}' \ar[r]  & \cdots \ar[r] & E_0' \ar[r]  & A \ar[r] & 0 
\end{tikzcd}
\right).
\]
That is, two extensions $E$ and $E'$ are \textit{congruent} if there exists extensions $E^{(i)}$ for $0 \le i \le n$ such that $E = E^{(0)}$, $E' = E^{(n)}$ and, for $i \le n-1$, there exists an elementary congruence of the form $\varphi: E^{(i)} \to E^{(i+1)}$ or $\varphi : E^{(i+1)} \to E^{(i)}$.

We write extensions in $\Ext_{R}^{n}(A,B)$ as $E=(E_*,\partial_*)$ where the maps $i : B \to E_{n-1}$ and $\varepsilon: E_0 \to A$ are understood. We will often write $\partial_i = \partial_i^E$, $i=i_E$ and $\varepsilon = \varepsilon_E$ when the need arises to distinguish different extensions.

This is an abelian group under Baer sum, and coincides with the usual definition of $\Ext_{R}^n(A,B)$ \cite[Section 3.4]{We94}. 
We will assume familiarity with the basic operations on extensions such as pullback, pushout and Yoneda product \cite[Section 24]{Jo03a}.

Worth emphasising however is the operation of stabilisation. If $E = (E_*,\partial_*) \in \Ext^n_{R}(A,B)$, then define the \textit{stabilised complex} $E \oplus R \in \Ext_{R}^{n}(A, B \oplus R)$ by
\[
	E \oplus R =( 0 \to B \oplus R \xrightarrow[]{\cdot\text{$\left( \begin{smallmatrix} i & 0 \\ 0 & 1 \end{smallmatrix}\right)$}} E_{n-1} \oplus R \xrightarrow[]{\cdot\text{$\left( \begin{smallmatrix} \partial_{n - 1} \\ 0 \end{smallmatrix}\right)$}} E_{n-2} \to \cdots \to E_0 \to A \to 0).
\]
This gives a well-defined map of abelian groups
\[ - \oplus R : \Ext^n_{R}(A,B) \to \Ext^{n}_{R}(A,B \oplus R).\]

Let $\Proj_{R}^{n}(A,B)$ denote the subset of $\Ext_{R}^n(A,B)$ consisting of extensions $(P_*,\partial_*)$ with the $P_i$ projective. This is closed under Baer sum, and so is a subgroup, and is also preserved by pullbacks, pushouts, Yoneda product and stabilisation.
The following is a consequence of the co-cycle description of $\Ext$ \cite[Lemma 1.1]{Wa79a}.

\begin{lemma}[Shifting] \label{lemma:yoneda}
If $A, B, C, D$ are $R$-modules, $E \in \Proj_{R}^k(B,C)$ and $k,n,m \ge 1$, then Yoneda product induces bijections
\[ - \circ E : \Ext_{R}^{n}(C,D) \to \Ext_{R}^{n+k}(B,D), 
\quad E \circ - : \Ext_{R}^{m}(A,B) \to \Ext_{R}^{m+k}(A,C).\]
\end{lemma}

This can be viewed as a sort of cancellation theorem for extensions up to congruence in the sense that $F \circ E \cong F' \circ E$ or $E \circ F \cong E \circ F'$ implies that $F \cong F'$.

A simple consequence of this is the following lemma. This can be interpreted as a kind of duality theorem for projective extensions.

\begin{lemma}[Duality] \label{lemma:yoneda-duality}
If $A, B, C$ are $R$-modules, $F \in \Proj_{R}^k(A,C)$ and $k > n \ge 1$, then there are bijections
\begin{align*} \Psi_F : & \,\, \Proj_{R}^{n}(A,B) \to \Proj_{R}^{k-n}(B,C),  &\Psi_F^{-1} : & \,\, \Proj_{R}^{k-n}(B,C) \to \Proj_{R}^{n}(A,B).\\
  & \,\,E \mapsto (- \circ E)^{-1}(F) && \,\,E' \mapsto (E' \circ -)^{-1}(F)\end{align*}
\end{lemma}

We now turn our attention to an equivalence relation on $\Ext_{R}^n(A,B)$ which is weaker than congruence. For $R$-modules $A, B$ and $E, E' \in \Ext_{R}^n(A,B)$, a chain map $\varphi : E \to E'$ is said to be a \textit{chain homotopy equivalence} if the restriction to the un-augmented chain complexes $\varphi: (E_*,\partial_*)_{0 \le * <n} \to (E'_*,\partial_*')_{0 \le * <n}$ is a chain homotopy equivalence. 

If $E, E' \in \Proj_{R}^n(A,B)$ then, since a chain map between projective chain complexes is a chain homotopy equivalence if and only if it is a homology equivalence \cite[Theorem 46.6]{Jo03a}, a chain homotopy equivalence $\varphi : E \to E'$ can equivalently be defined as a chain map of the form
\[
\begin{tikzcd}[row sep=.5cm, column sep=small]
E \ar[d,"\varphi"] \\
E'
\end{tikzcd}
= \left( 
\begin{tikzcd}[row sep=.5cm, column sep=small]
0 \ar[r] & B \ar[r] \ar[d,"\varphi_B"] & P_{n-1} \ar[r] \ar[d,"\varphi_{n-1}"] & \cdots \ar[r] & P_0 \ar[r] \ar[d,"\varphi_0"] & A \ar[r] 
\ar[d,"\varphi_A"] & 0  \\
0 \ar[r] & B \ar[r] & P_{n-1}' \ar[r]  & \cdots \ar[r] & P_0' \ar[r]  & A \ar[r] & 0 
\end{tikzcd}
\right)
\]
where $\varphi_A$ and $\varphi_B$ are $R$-module isomorphisms. When convenient, we will often abbreviate this to $\varphi=(\varphi_B,\varphi_{n-1},\cdots,\varphi_0,\varphi_A)$.
It follows easily that a congruence is a chain homotopy equivalence. We define $\hProj_{R}^n(A,B)$ to be set of equivalence classes in $\Proj_{R}^n(A,B)$ up to chain homotopy equivalences, which is an abelian group under Baer sum.

For special choices of modules, the shifting lemma and the duality lemma also hold for chain homotopy equivalences. We define $\Z$ to be the $R$-module with underlying abelian group $\Z$ and trivial $R$-action, i.e. $r \cdot n = n$ for all $r \in R, n \in \Z$.

\begin{lemma}[Shifting] \label{lemma:h-yoneda}
If $A, B$ are $R$-modules, $F \in \Proj_{R}^k(\Z,\Z)$ and $n,m,k \ge 1$, then Yoneda product induces bijections
\[ - \circ F : \hProj_{R}^{n}(\Z, A) \to \hProj_{R}^{n+k}(\Z,A),\quad F \circ - : \hProj_{R}^{m}(B,\Z) \to \hProj_{R}^{m+k}(B,\Z).\]
\end{lemma}

\begin{proof}
Firstly note that $- \circ F$ induces maps on the chain homotopy classes by extending the map to $\pm \id$ on $F$. This is necessarily surjective. To see that it is injective, suppose that there is a chain homotopy equivalence $\varphi: E_1 \circ F \to E_2 \circ F$. By considering $-\varphi$ if necessary, we can assume that $\varphi_{\Z} = \id$, so that 
\[ E_2 \circ F \cong (\varphi_A)_*(E_1 \circ F) = (\varphi_A)_*(E_1) \circ F.\]
By Lemma \ref{lemma:yoneda}, this implies that $E_2 \cong (\varphi_A)_*(E_1)$ and so $E_1 \simeq E_2$ as required.	
\end{proof}

The proof of the duality lemma in this setting is similar as so will be omitted. 

\begin{lemma}[Duality] \label{lemma:h-yoneda-duality}
If $A$ is an $R$-module, $F \in \Proj_{R}^{k}(\Z,\Z)$ and $k > n \ge 1$, then there are bijections
	\begin{align*} \Psi_F : & \,\, \hProj_{R}^{n}(\Z,A) \to \hProj_{R}^{k-n}(A,\Z),  &\Psi_F^{-1} : & \,\, \hProj_{R}^{k-n}(A,\Z) \to \hProj_{R}^{n}(\Z,A).\\
  & \,\,E \mapsto (- \circ E)^{-1}(F) && \,\,E' \mapsto (E' \circ -)^{-1}(F)\end{align*}
\end{lemma}

We now specialise to the case where the underlying abelian group of $R$ is finitely generated and torsion-free, and where $R$ is a ring with involution, i.e. a ring with an anti-automorphism $r \mapsto \bar{r}$ such that $\bar{\bar{r}} = r$ for all $r \in R$.
For example, for a finite group $G$, the group ring $\Z G$ has underlying abelian group $\Z^{|G|}$ and involution $\sum_{i=1}^n n_i g_i \mapsto \sum_{i=1}^n n_i g_i^{-1}$ where $n_i \in \Z$, $g_i \in G$. 
Using this involution, any right $R$-module $A$ can be viewed as a left $R$-module under the action $r \cdot x = x \cdot \bar{r}$ for $r \in R$, $x \in A$.
If $A$ is a left $R$-module, then $A^* = \Hom_{R}(A,R)$ is a right $R$-module under the action $(\varphi \cdot r) (x) = \varphi(x)r$ for $\varphi \in A^*$, $r \in R$. We will view $A^*$ as a left $R$-module using the involution on $R$.

Note that $( \,\cdot\, )^*$ can be viewed as a functor of $R$-modules: if $f : A_1 \to A_2$ is a map of $R$-modules, we can define $f^* : A_2^* \to A_1^*$ by sending $\varphi \mapsto \varphi \circ f$.
For $E = (P_*,\partial_*) \in \Proj^n_{R}(A,B)$, define the \textit{dual extension} by
\[ E^* = ( 0 \to A^* \xrightarrow[]{\varepsilon^*} P_{0}^* \xrightarrow[]{\partial_{1}^*} P_{1}^* \xrightarrow[]{\partial_{2}^*} \cdots \xrightarrow[]{\partial_{n-2}^*} P_{n-2}^* \xrightarrow[]{\partial_{n-1}^*} P_{n-1}^* \xrightarrow[]{i^*} B^* \to 0).\]
The dual of a projective module is projective since $P \oplus Q \cong R^n$ implies that $P ^* \oplus Q^* \cong (R^n)^* \cong R^n$. In particular, the $P_i^*$ are projective $R$-modules.

Whilst $E^*$ is not exact in general, it is true under mild assumptions on the modules involved. We say that an $R$-module $A$ is an \textit{$R$-lattice} if its underlying abelian group is finitely generated and torsion-free.
For example, if $P$ is a (finitely generated) projective $R$-modules, then $P$ is an $R$-lattice. This follows from the fact that $P \le R^n$ is an $R$-submodule for some $n$ and so its underlying abelian group is a subgroup of $\Z^{m}$ where $m = n \cdot \rank_{\Z}(R)$.

Recall that the \textit{evaluation map} is the map $e_A : A \to A^{**}$, $x \mapsto (f \mapsto f(x))$. We say an $R$-module is \textit{reflexive} if $e_A$ is an $R$-module isomorphism.

\begin{lemma} \label{lemma:lattice**=lattice}
If $A$ is a $R$-lattice, then $A$ is reflexive.
\end{lemma}

\begin{remark}
Since projective $R$-modules are $R$-lattices, this implies that they are reflexive. We note that this is true for arbitrary rings $R$, not just rings with involution whose underlying abelian group is finitely generated and torsion free.
\end{remark}

This follows by noting that, if $A \cong_{\text{Ab}} \Z^k$, then the $R$-module structure is determined by a map $\rho_A: R \to M_k(\Z)$. It can be shown that $\rho_{A^*}(r) = \rho_A(\bar{r})^T$ using the induced identification $A^* \cong_{\text{Ab}} \Z^k$, from which the claim follows. 

It follows easily from this that the reflexivity property of $R$-lattices also holds on the level of extensions.

\begin{lemma}[Reflexivity] \label{lemma:double-dual}
If $A$, $B$ are $R$-lattices and $n \ge 1$, then dualising gives an isomorphism of abelian groups
\[ * :  \hProj^n_{R}(A,B) \to \hProj^n_{R}(B^*,A^*).\]
If $E \in \Proj^n_{R}(A,B)$, then there is a chain homotopy equivalence $e: E \to E^{**}$ induced by the evaluation maps.
\end{lemma}

This has the following useful consequence which, in the language of \cite[Theorem 28.5]{Jo03a}, says that projective $R$-modules are \textit{injective relative to the class of $R$-lattices}.

\begin{lemma} \label{lemma:relative-inj}
Suppose $A$, $B$ and $E$ are $R$-lattices such that $(E,-) \in \Ext^1_{R}(A,B)$ and $P$ is a projective $R$-module. Then, for any map $f: B \to P$, there exists $\widetilde{f} : E \to P$ such that $\widetilde{f} \circ i = f$, i.e.
\[
\begin{tikzcd}
0 \ar[r] & B \ar[d,"f"] \ar[r,"i"] & E \ar[r,"\varepsilon"] \ar[dl,dashed,"\widetilde{f}"] & A \ar[r] & 0\\
& P & & &	
\end{tikzcd}
\]
\end{lemma}

We conclude this section by discussing an important invariant of projective extensions.
Let $P(R)$ denote the $R$-module isomorphism classes of (finitely generated) projective $R$-modules and define the \textit{projective class group} $C(R)$ as the quotient of $P(R)$ by the stable isomorphisms, where $P, Q \in P(R)$ are \textit{stably isomorphic}, written $[P] = [Q]$, if $P \oplus R^i \cong Q \oplus R^j$ for some $i, j \ge 0$. This forms a group under direct sum and coincides with the Grothendieck group of the monoid $P(R)$.

For a projective extension
\[ E = ( 0 \to B \xrightarrow[]{i} P_{n-1} \xrightarrow[]{\partial_{n-1}} P_{n-2} \xrightarrow[]{\partial_{n-2}} \cdots \xrightarrow[]{\partial_{2}} P_1 \xrightarrow[]{\partial_{1}} P_0 \xrightarrow[]{\varepsilon} A \to 0),\]
we define the \textit{Euler class} $e(E) = \sum_{i=0}^{n-1} (-1)^i [P_i] \in C(R)$. This is known to be a congruence invariant \cite[Lemma 1.3]{Wa79a}. In fact, more is true:

\begin{lemma}
If $A, B$ are $R$-modules, the Euler class defines a map
\[ e : \hProj_{R}^n(A,B) \to C(R),\]
i.e. $e$ is a chain homotopy invariant.
\end{lemma}

\begin{proof}
Suppose $E_1, E_2 \in \Proj_{R}^n(A,B)$ and that $\varphi: E_1 \to E_2$ is a chain homotopy equivalence. Then $E_2 \cong (\varphi_A)^*((\varphi_B)_*(E_1))$ and, since $e$ is a congruence invariant, we have that $e(E_2) = e((\varphi_A)^*((\varphi_B)_*(E_1)))$. Since pushout and pullback by automorphisms can be made to not affect the isomorphism classes of the modules in the extension, this implies that $e((\varphi_A)^*((\varphi_B)_*(E_1))) = e(E_1)$ and so $e$ is a chain homotopy invariant.	
\end{proof}

The following tells us how the Euler class interacts with with the Yoneda product.

\begin{lemma} \label{lemma:e-of-prod}
Let $A$, $B$ and $C$ be $R$-modules. If $E \in \hProj_R^n(A,B)$ and $F \in \hProj_R^m(B,C)$, then
\[ e(F \circ E) = e(E) + (-1)^n e(F). \]
\end{lemma}

\begin{proof}
	Let $E = (P_*, \partial_*)_{*=0}^{n-1}$ and let $F = (P_{*+n},\partial_{*+n})_{*=0}^{m-1}$. Then $F \circ E = (P_*,\partial_*)_{*=0}^{n+m-1}$ and
	\[ e(F \circ E) = \sum_{i=0}^{n+m-1} (-1)^i[P_i] = \sum_{i=0}^{n-1} (-1)^i[P_i] + \sum_{i=0}^{m-1} (-1)^{i+n}[P_{i+n}] = e(E) + (-1)^n e(F). \qedhere \]
\end{proof}

For a class $\chi \in C(R)$, we define $\Proj_{R}^{n}(A,B;\chi)$ to be the subset of $\Proj_{R}^{n}(A,B)$ consisting of those extensions with $e(E) = \chi$, and we can define $\hProj_{R}^n(A,B;\chi)$ similarly as a subset of $\hProj_{R}^n(A,B)$.

We have the following nice interpretations for the extensions $E \in \Proj^n_{R}(A,B)$ with $e(E) =0$. This follows easily by repeatedly forming the direct sum with length two extensions $P \xrightarrow[]{\cong} P$ for various $P \in P(R)$.

\begin{lemma} \label{lemma:free-res}
If $A, B$ are $R$-modules and $n \ge 2$, then every congruence class in $\Proj_{R}^n(A,B;0)$ has a representative $E$ of the form $E = (F_*, \partial_*)$ with the $F_i$ free.
\end{lemma}

This fails in the case $n=1$, where it is not possible to form the direct sum with length two extensions $R \xrightarrow[]{\cong} R$ without altering the chain homotopy type. In fact, for a projective extension
\[ E = (0 \to B \to P \to A \to 0 ),\]
we can define the \textit{unstable Euler class} $\wh e(E) = P \in P(R)$.

\begin{lemma} \label{lemma:unstable-euler}
If $A, B$ are $R$-modules, the unstable Euler class defines a map
\[ \wh e :  \hProj_{R}^1(A,B) \to P(R).\]
\end{lemma}

\begin{proof}
For $E_1=(P_1,-), E_2 = (P_2,-) \in \Proj_{R}^1(A,B)$, recall that a chain map $\varphi: E_1 \to E_2$ is a chain homotopy equivalence if it induces a chain homotopy equivalence between the length one chain complexes $P_1$ and $P_2$, i.e. if the restriction $\varphi \mid_{P_1} : P_1 \to P_2$ is an isomorphism.
\end{proof}


\section{Projective $\Z G$-modules and the Swan finiteness obstruction}
\label{section:swan-modules-and-projective-extensions}

Throughout this section, we will let $G$ be a finite group. The results of the previous section apply in the case $R=\Z G$ since $\Z G$ is a ring with involution which is finitely generated and torsion-free as an abelian group.
The aim of this section will be to recall some of the special features of projective modules over $\Z G$ and to introduce the Swan finiteness obstruction.

\subsection{Preliminaries on projective $\Z G$-modules} \label{ss:ZG-mod}

We will now summarise the main special properties of (finitely generated) projective $\Z G$-modules in the case where $G$ is finite. 

The first was shown by Swan in \cite[Theorem A]{Sw60b}.

\begin{prop} \label{prop:swan-induced}
Let $P$ be a projective $\Z G$-module. Then there is a projective ideal $I \subseteq \Z G$ such that $P \cong I \oplus \Z G^r$ for some $r \ge 0$.
\end{prop}

For a prime $p$, let $\Z_p$ denote the $p$-adic integers and let $\Z_{(p)} = \{ \frac{a}{b} : a,b \in \Z, p \nmid b\} \le \Q$ denotes the localisation at $p$. 
The next property that projective modules over $\Z G$ have is that they are locally free in the following sense (see \cite[Section 2]{Sw80} for further discussion).

\begin{prop} \label{prop:LF}
Let $P$ be a projective $\Z G$-module. There exists $n \ge 0$ such that
\begin{enumerate}[\normalfont (i)]
	\item $P \otimes \Z_{(p)} \cong \Z_{(p)} G^n$ are isomorphic as $\Z_{(p)} G$-modules
	\item $P \otimes \Q \cong \Q G^n$ are isomorphic as $\Q G$-modules
	\item $P \otimes \Z_p \cong \Z_p G^n$ are isomorphic as $\Z_p G$-modules.
	\item $P \otimes \Q_p \cong \Q_p G^n$ are isomorphic as $\Q_p G$-modules
\end{enumerate}
\end{prop}

\begin{proof}
(ii) and (iv) each follow from \cite[Theorem 4.2]{Sw70}. Given this, (i) and (iii) now follow from \cite[Theorem 2.21]{Sw70}.
\end{proof}

We define the \textit{rank} of $P$, denoted by $\rank(P)$, to be the $n \ge 0$ in the proposition above. 
For example, if $I \subseteq \Z G$ is a non-zero projective ideal, then it can be shown that $\rank(I) = 1$ (see \cite[Section 7]{Sw60b}).

Let $P(\Z G)$ denote the set of $\Z G$-module isomorphism classes of non-zero projective $\Z G$-modules. This is a monoid under direct sum.
Since $\rank(P \oplus Q) =\rank(P)+\rank(Q)$ for all $P,Q \in P(\Z G)$, we have that there is a surjective homomorphism of monoids
\[ \rank : P(\Z G) \to \Z, \quad P \mapsto \rank(P).\]

Note that $\rank(P) = 0$ if and only if $P=0$. That is, if $P$ is a non-zero projective $\Z G$-module, then $\rank(P) \ge 1$. This has the following consequence.

\begin{corollary} \label{cor:proj-surj}
Let $P$ be a non-zero projective $\Z G$-module. Then there exists a surjection $\varphi: P \to \Z$.
\end{corollary}

\begin{proof}
	Let $n = \rank(P) \ge 1$ and consider the composition
	\[ P \xhookrightarrow[]{x \mapsto x \otimes 1} P \otimes \Q \xrightarrow[]{\cong} \Q G^n \xrightarrowdbl[]{\pi_1} \Q G \xrightarrowdbl[]{\varepsilon} \Q \]
	where $\pi_1$ is projection onto the first coordinate and $\varepsilon$ is the augmentation map. Since $P$ is finitely generated, the image of the composition is a finitely generated subgroup of $\Q$ and so is isomorphic to $\Z$. This gives the required surjection.
\end{proof}

\subsection{Swan modules} \label{ss:swan-modules}

We will now define Swan modules which are a special type of projective module first introduced in \cite[Section 6]{Sw60a}. 
Let $\varepsilon: \Z G \to \Z$ denote the augmentation map and let $I = \Ker(\varepsilon) \subseteq \Z G$ denote the the augmentation ideal.
For any $r \in \Z$ coprime to $|G|$, the ideal $(I,r) \subseteq \Z G$ is projective and depends only on $r$ mod $|G|$ up to $\Z G$-isomorphism \cite{Sw60a}. Since $(I,r)$ is a non-zero ideal, it has rank one as a projective $\Z G$-module by the remarks in Section \ref{ss:ZG-mod}.

The modules $(I,r)$ are known as \textit{Swan modules} and the map
\[ S: (\Z / |G|)^\times \to C(\Z G)\]
given by $r \mapsto [(I,r)]$ is known as the \textit{Swan map}. This is a well-defined group homomorphism \cite{Sw60a}, and we define the \textit{Swan subgroup} to be $T_G = \IM(S) \le C(\Z G)$.

Whilst we will not make explicit use of it in this article, we will briefly mention the closely related ideal $(N,r) \subseteq \Z G$ where $N = \sum_{g \in G} g$ denotes the group norm. 
Many authors take the $(N,r)$ to be Swan modules instead of the ideals $(I,r)$. In fact, the two notions are equivalent, as the following proposition shows.

\begin{prop} \label{lemma:swan=swan*}
If $G$ is a finite group and $r \in (\Z/|G|)^\times$, then $(I,r) \cong (N, r^{-1})$.	
\end{prop}

This is presumably well known, but we will include a detailed proof here since we are not aware that one is currently available in the literature.

\begin{proof}
By the uniqueness of pullbacks, it will suffice to prove that both $(I,r)$ and $(N, r^{-1})$ arise as pullbacks of the map $r: \Z \to \Z/|G|$ which sends $1 \mapsto r$, and the map $\varepsilon: \Z G/(N) \to \Z /|G|$ which sends $x + (N) \mapsto \varepsilon(x) + |G|$.

First let $i : I \hookrightarrow (I,r)$ denote inclusion, let $\varphi: (I,r) \to \Z G/ (N)$ and let $q: \Z G \twoheadrightarrow \Z G/(N)$ denote the quotient map. Then there is a diagram
\[
\begin{tikzcd}[row sep=.5cm]
	0 \ar[r] & I \ar[r,"i"] \ar[d,"\id"] & (I,r) \ar[r,"\frac{1}{r}\varepsilon"] \ar[d,"q"] & \Z \ar[r] \ar[d,"r"] & 0 \\
	0 \ar[r] & I \ar[r,"j"] & \Z G/(N) \ar[r,"\varepsilon"]& \Z/|G| \ar[r] & 0
\end{tikzcd}
\]
where $q$ and $\frac{1}{r}\varepsilon$ denote the restrictions of these maps to $(I,r) \subseteq \Z G$ and $j = q \circ i$. It can be checked that the diagram commutes and that the rows are exact, and so the right hand square is a pullback.

Now let $s \in \Z$ be such that $s=r^{-1} \in (\Z /|G|)^\times$, so that $(N, r^{-1}) \cong (N, s)$. Define $f: (N,s) \to \Z G/(N)$ by sending $N x + s y \mapsto y$. Then consider the diagram
\[
\begin{tikzcd}[row sep=.5cm]
	0 \ar[r] & I \ar[r,"s"] \ar[d,"\id"] & (N,s) \ar[r,"\varepsilon"] \ar[d,"f"] & \Z \ar[r] \ar[d,"r"] & 0 \\
	0 \ar[r] & I \ar[r,"j"] & \Z G/(N) \ar[r,"\varepsilon"]& \Z/|G| \ar[r] & 0.
\end{tikzcd}
\]
Similarly, it can be checked that this commutes and that the rows are exact.
\end{proof}

\subsection{Projective extensions} \label{ss:proj-ext}

We will now consider the classification of extensions $\Proj^n_{\Z G}(\Z,A)$ for a fixed $\Z G$-module $A$. 
The following can be found in \cite[Proposition 34.2]{Jo03a} and shows that any two elements of $\Proj^n_{\Z G}(\Z,A)$ are related by pullbacks. Note that this isomorphism depends on the choice of $E$ and so only exists when $\Proj^n_{\Z G}(\Z,A)$ is non-empty.

\begin{prop} \label{prop:congruence-classification}
Let $A$ be a $\Z G$-module and $n \ge 1$. Then, for any $E \in \Proj^n_{\Z G}(\Z,A)$, there is a bijection
\[ (m_{\cdot})^*: (\Z /|G|)^\times \to \Proj^n_{\Z G}(\Z,A)\]
given by $r \mapsto (m_r)^*(E)$, where $m_r: \Z \to \Z$ denotes multiplication by $r$.
\end{prop}

\begin{remark}
This corresponds to the fact that extensions with fixed ends are determined by their $k$-invariants (see, for example, \cite[Chapter 6]{Jo03a}).
\end{remark}

Let $e$ denote the stable Euler class as defined in Section \ref{section:ext}. The next result computes the image of projective extensions under the stable Euler class.

\begin{prop} \label{prop:image-of-chi}
Let $e$ denote the stable Euler class. Let $n \ge 1$ and let $A$ be a $\Z G$-module such that there exists $E \in \Proj^n_{\Z G}(\Z,A)$.
If $e(E)=[P]$, then
\[ e( \Proj^n_{\Z G}(\Z,A)) = [P] + T_G \subseteq C(\Z G).\]
\end{prop}

\begin{proof}
This was proven in \cite[Lemmas 7.3 \& 7.4]{Sw60a} in the case $A=\Z$ and the proof for arbitrary $A$ is analogous. We will outline the steps here for the convenience of the reader.
 
The first step is to show that, for any $E, E' \in \Proj_{\Z G}^n(\Z, A)$, we have $e(E')-e(E) \in T_G$. By applying Schanuel's lemma (see \cite[Proposition 1.1]{Sw60a}) to the duals $E^*, (E')^* \in \Proj_{\Z G}^n(A^*,\Z)$, we get an isomorphism $\Z \oplus e(E^*) \cong \Z \oplus e((E')^*)$ and so $e((E')^*)-e(E^*) \in T_G$ by \cite[Lemma 6.2]{Sw60a}. Since $e(E^*) = e(E)^*$ and projective $\Z G$-modules are reflexive, dualising gives that $e(E')-e(E) \in T_G$.

	The second step is to show that, given $E \in \Proj_{\Z G}^n(\Z, A)$, there exists $E' \in \Proj_{\Z G}^n(\Z, A)$ such that $e(E')-e(E) = [(I,r)]$. This can be constructed in the same way as in \cite[Lemma 7.4]{Sw60a}. That is, using \cite[Remark 2.1]{Sw60a}.
\end{proof}

\subsection{The Swan finiteness obstruction} \label{ss:SFO}

We will now specialise further to the case $A = \Z$.
Recall that a finite group $G$ is said to have \textit{$k$-periodic cohomology} if there is an isomorphism of abelian groups $\wh H^i(G;\Z) \cong \wh H^{i+k}(G;\Z)$ for all $i \in \Z$.

\begin{remark}
Many authors define finite groups with periodic cohomology by the a priori stronger condition that there exists a class $u \in \wh H^k(G;\Z)$ such that cup product induces an isomorphism 
\[ u \cup - : \wh H^i(G;\Z) \to \wh H^{i+k}(G;\Z)\]
 for all $i \in \Z$. These definitions are equivalent since, if $\wh H^i(G;\Z) \cong \wh H^{i+k}(G;\Z)$ for all $i \in \Z$, then $\wh H^k(G;\Z) \cong \wh H^0(G;\Z) \cong \Z/|G|$ which implies that the conditions above holds by \cite[VI.9.1]{Br82}.
\end{remark}

The following can be extracted from \cite[Chapter XII]{CE56}.

\begin{prop}
Let $G$ be a finite group. Then $G$ has $k$-periodic cohomology if and only if $\Proj_{\Z G}^k(\Z,\Z)$ is non-empty.
\end{prop}

If $G$ has $k$-periodic cohomology then, since $\Proj^k_{\Z G}(\Z,\Z)$ is non-empty, Proposition \ref{prop:image-of-chi} implies that there exists $P \in P(\Z G)$ for which
\[ e(\Proj^k_{\Z G}(\Z,\Z)) = [P] + T_G \subseteq C(\Z G)\]
where $P(\Z G)$ denotes the set of non-zero projective $\Z G$-modules.
We can then quotient by $T_G$ to get a unique class in $C(\Z G)/T_G$ which depends only on $G$ and $k$. The \textit{Swan finiteness obstruction} is defined as
\[ \sigma_k(G)=[P] \in C(\Z G)/ T_G.\]

Recall that a group $G$ has \textit{free period $k$} if there exists $E = (F_*,\partial_*) \in \Proj_{\Z G}^k(\Z, \Z)$ with the $F_i$ free.
The following is \cite[Proposition 5.1]{Sw60a}.

\begin{prop}
Let $G$ have $k$-periodic cohomology. Then $\sigma_k(G) = 0$ if and only if $G$ has free period $k$. 
\end{prop}

\begin{remark}
By a construction of Milnor, this is equivalent to the existence of a finite CW-complex $X$ with $\pi_1(X) \cong G$ and $\widetilde{X} \simeq S^{k-1}$ \cite[Proposition 3.1]{Sw60a}.
Examples of groups with $\sigma_k(G) \ne 0$ were found by Milgram \cite{Mi85}.
\end{remark}

We will conclude this section by giving a constraint on the projective $\Z G$-modules $P$ which can arise as a representative of the Swan finiteness obstruction.

We would like to compare $[P]$ and $[P^*]$ when $\sigma_k(G) = [P] +T_G$. This is difficult for general projectives since there exists finite groups $G$ and projectives $P$ for which $[P^*] \ne \pm [P]$, even in $C(\Z G)/T_G$. For example, we can take $G = \Z /{37^2}$ \cite[Theorem 50.56]{CR87}. However, in our situation, we have the following.

\begin{prop} \label{prop:[P]=-[P^*]-mod-T_G}
If $G$ has $k$-periodic cohomology, and $\sigma_k(G) = [P] + T_G$, then
\[ [P]=-[P^*] \in C(\Z G)/T_G.\]	
\end{prop}

\begin{proof}
By Proposition \ref{prop:image-of-chi}, there exists $E \in \Proj^k_{\Z G}(\Z,\Z)$ with $e(E) = [P]$ and, by forming the direct sum with length two extensions $\Z G \xrightarrow[]{\cong} \Z G$, we can assume that
\[ E \cong (0 \to \Z \xrightarrow[]{i} P \xrightarrow[]{\partial_{k-1}} F_{k-2} \xrightarrow[]{\partial_{k-2}} \cdots \xrightarrow[]{\partial_{1}} F_0 \xrightarrow[]{\varepsilon} \Z \to 0)\]
for some $F_i$ free. Dualising then gives that
\[ E^* \cong (0 \to \Z \xrightarrow[]{\varepsilon^*} F_0 \xrightarrow[]{\partial_{1}^*}  \cdots \xrightarrow[]{\partial_{k-2}^*} F_{k-2} \xrightarrow[]{\partial_{k-1}^*} P^* \xrightarrow[]{i^*} \Z \to 0)\]
and, since $k$ is necessarily even \cite[p261]{CE56}, Schanuel's lemma implies that
\[ \Z \oplus P \oplus P^* \oplus F \cong \Z \oplus F'\]
for some $F$, $F'$ free. By \cite[Lemma 6.2]{Sw60a}, we then get that $[P \oplus P^*] \in T_G$.	
\end{proof}

\begin{remark}
For a finite group $G$, the standard involution on $C(\Z G)$ is given by $[P] \mapsto -[P^*]$ (see \cite[Section 50E]{CR87}). This turns $C(\Z G)$ into a $\Z C_2$-module where the $C_2$-action is given by the involution.
This additional structure has proven to be a useful for computing class groups \cite[p284]{CR87}.
Note that $T_G$ is fixed by this involution. This follows from the fact that $(I,r)^* \cong (N,r) \cong (I,r^{-1})$ by \cite[Lemma 17.1]{Sw83} and Lemma \ref{lemma:swan=swan*} respectively.
Hence the involution induces an involution on $C(\Z G)/T_G$ and so endows it with a natural $\Z C_2$-module structure. With respect to this action, Proposition \ref{prop:[P]=-[P^*]-mod-T_G} says that $\sigma_k(G) \in (C(\Z G)/T_G)^{C_2}$.
\end{remark}


\section{Classification of projective chain complexes}
\label{section:cancellation-of-projective-extensions}

We would now like to consider more generally the classification of projective extensions over $\Z G$ with only one fixed end. Throughout this section, $G$ will denote a finite group. For $n \ge 0$, a \textit{projective $n$-complex} $E=(P_*,\partial_*)$ over $\Z G$ is a chain complex consisting of an exact sequence
\[ E =( P_n \xrightarrow[]{\partial_n} P_{n-1} \xrightarrow[]{\partial_{n-1}} \cdots \xrightarrow[]{\partial_1}P_0 )\]
where $H_0(P_*) \cong \Z$ and the $P_i$ are (finitely generated) projective $\Z G$-modules. 
An \textit{algebraic $n$-complex} is a projective $n$-complex such that the $P_i$ are free.

Let $\Proj(G,n)$ denote the set of chain homotopy types of projective $n$-complexes over $\Z G$, which is a graded graph with edges between each $E=(P_*, \partial_*)$ and
\[ E \oplus \Z G =( P_n \oplus \Z G \xrightarrow[]{(\partial_n,0)} P_{n-1} \xrightarrow[]{\partial_{n-1}} \cdots \xrightarrow[]{\partial_1}P_0 ).\]
Similarly, let $\Alg(G,n)$ denote the set of chain homotopy types of algebraic $n$-complexes over $\Z G$, which is also a graded graph under stabilisation.
By extending the projective $n$-complex by $\Ker(\partial_n)$, it is easy to see that there is a bijection
\[ \Proj(G,n) \cong \coprod_{A \in \Mod(\Z G)} \hProj_{\Z G}^{n+1}(\Z,A).\]
By abuse of notation, we will assume they are the same, i.e. that an extension $E \in \Proj(G,n)$ lies in $\hProj_{\Z G}^{n+1}(\Z,A)$ for some $A$.
For a class $\chi \in C(\Z G)$, let $\Proj(G,n;\chi)$ denote the subset of projective extensions $E$ with $e(E) = \chi$. Note that $\Alg(G,n) \cong \Proj(G,n;0)$ for $n \ge 2$.

\subsection{General classification of projective $n$-complexes} \label{ss:Proj(G,n)}

The following is well-known (see \cite[Theorem 1.1]{Ma07} or \cite[Proof of Lemma 8.12]{HPY13}).

\begin{thm} \label{thm:proj-stability}
If $n \ge 0$ and $\chi \in C(\Z G)$, then $\Proj(G,n;\chi)$ is a graded tree, i.e. if $E$, $E' \in \Proj(G,n)$ have $e(E) = e(E')$, then $E \oplus \Z G^i \simeq E' \oplus \Z G^j$ for some $i,j \ge 0$.
\end{thm}

We will now prove a cancellation theorem for projective $n$-complexes. Our proof will be modelled on Hambleton-Kreck's proof  that, if $X$ and $Y$ are finite $2$-complexes with finite fundamental group such that $X \simeq X_0 \vee S^2$ and $X \vee S^2 \simeq Y \vee S^2$, then $X \simeq Y$ \cite[Theorem B]{HK93}. This idea was applied to algebraic $2$-complexes in \cite{Ha18}.

If $A$ is a $\Z G$-module, then $x \in A$ is \textit{unimodular} if there exists a map $f: A \to \Z G$ such that $f(x)=1$. Let $\Um(A) \subseteq A$ denote the set of unimodular elements in $A$.

\begin{lemma}
Let $A$, $B$ be $\Z G$-modules. Then
\begin{enumerate}[\normalfont (i)]
\item If $\varphi : A \to B$ is an isomorphism, then $\varphi(\Um(A)) = \Um(B)$
\item $(0,1) \in \Um(A \oplus \Z G)$, i.e. if $\varphi: A \oplus \Z G \to B$ is an isomorphism, then $\varphi(0,1) \in \Um(B)$.
\end{enumerate}	
\end{lemma}

Suppose a $\Z G$-module $A$ has a splitting $A = A_1 \oplus A_2 \oplus \cdots \oplus A_n$. Then a map $f: A_i \to A_j$ can be viewed as an endomorphism of $A$ by extending it to vanish everywhere else. Write $\GL(A)$ for the group of automorphisms of $A$ and define
\[ E(A_i,A_j) = \langle 1+f, 1+g: f:A_i \to A_j, g:A_j \to A_i \rangle \le \GL(A)\]
to be the subgroup of \textit{elementary automorphisms} for $i \ne j$, where $1: A \to A$ denotes the identity map.

The main result we will use is the following, which can be proven by combining \cite[Corollary 1.12]{HK93} with \cite[Lemma 1.16]{HK93}. Let $\Z_{(p)} = \{ \frac{a}{b} : a,b \in \Z, p \nmid b\} \le \Q$ denotes the localisation at a prime $p$ and $A_{(p)} = A \otimes \Z_{(p)}$. 

\begin{thm} \label{thm:HK}
Suppose $A$ is a $\Z G$-module for which $\Z_{(p)} \oplus A_{(p)}$ is a free $\Z_{(p)} G$-module for all but finitely many primes $p$. If $F_1$, $F_2 \cong \Z G$, then 
\[ \mathcal{G} = \langle E(F_1,A\oplus F_2), E(F_2,A\oplus F_1) \rangle \le \GL(A \oplus F_1 \oplus F_2)\]
acts transitively on $\Um(A \oplus F_1 \oplus F_2)$.
\end{thm}

We will now establish criteria for which the above conditions hold for a $\Z G$-module $A$. First recall that, by an extension of Maschke's theorem of representations, the group ring $R G$ is semisimple whenever $R$ is a commutative ring such that $|G| \in R^\times$. This is the case when $R= \Z_{(p)}$ for $p$ a prime not dividing $|G|$.
This has the following consequence.

\begin{lemma} \label{lemma:maschke}
Let $n \ge 1$ be odd, let $p$ be a prime not dividing $|G|$ and let $A$ be a $\Z G$-module for which $\Proj^n_{\Z G}(\Z,A)$ is non-empty. Then $\Z_{(p)} \oplus A_{(p)}$ is a free $\Z_{(p)} G$-module.
\end{lemma}

\begin{proof}
Let $E=(P_*,\partial_*) \in \Proj_{\Z G}^n(\Z,A)$. Recall that localisation is an exact functor (since, for example, $\Z_{(p)}$ is a flat module). Hence we obtain $E_{(p)} =((P_*)_{(p)}),\partial_*) \in \Proj_{\Z_{(p)} G}^n(\Z_{(p)},A_{(p)})$ where the $\partial_*$ are the induced maps. By the extension of Maschke's theorem mentioned above, $\Z_{(p)} G$ is semisimple and so the exact sequence $E_{(p)}$ splits completely. This implies that there is an isomorphism of $\Z_{(p)} G$-modules:
\[ \Z_{(p)} \oplus A_{(p)} \oplus \bigoplus_{\text{$i$ odd}} (P_i)_{(p)}  \cong \bigoplus_{\text{$i$ even}} (P_i)_{(p)}. \]

By Proposition \ref{prop:LF}, the $(P_i)_{(p)}$ are all free $\Z_{(p)} G$-modules. It follows that $\Z_{(p)} \oplus A_{(p)}$ is a stably free $\Z_{(p)} G$-module. Since $\Z_{(p)} G$ is semisimple, this implies that $\Z_{(p)} \oplus A_{(p)}$ is a free $\Z_{(p)} G$-module.
\end{proof}

Note that the fact that $\GL(A \oplus \Z G^2)$ acts transitively on $\Um(A \oplus \Z G^2)$ already implies the following cancellation theorem for modules.

\begin{corollary} \label{cor:HK}
Suppose $A$ is a $\Z G$-module, $A \cong A_0 \oplus \Z G$ and $\Z_{(p)} \oplus (A_0)_{(p)}$ is a free $\Z_{(p)} G$-module for all but finitely many primes $p$. Then $A \oplus \Z G \cong A' \oplus \Z G$ implies $A \cong A'$.
\end{corollary}

\begin{proof}
Let $\psi: A \oplus \Z G \to A' \oplus \Z G$ be an isomorphism and let $x = \psi^{-1}(0,1) \in \Um(A \oplus \Z G)$. Since $A= A_0 \oplus \Z G$, Theorem \ref{thm:HK} implies that $\GL(A \oplus \Z G)$ acts transitively on $\Um(A \oplus \Z G)$ and so there is an isomorphism $\varphi: A \oplus \Z G \to A \oplus \Z G$ such that $\varphi(0,1) = x$. Hence $\psi \circ \varphi: A \oplus \Z G \to A' \oplus \Z G$ has $(\psi \circ \varphi)(0,1)=(0,1)$ and so induces an isomorphism $(\psi \circ \varphi) \mid_A : A \to A' \oplus \Z G/\IM(0 \oplus \Z G) \cong A'$.
\end{proof}

We will upgrade the above argument from modules to projective $n$-complexes. The existence of a well-understood subgroup $\mathcal{G} \le \GL(A \oplus \Z G^2)$ which acts transitively on $\Um(A \oplus \Z G^2)$ is important since we need only show that elements in $\mathcal{G}$ can be extended to chain homotopy equivalences on the short exact sequences.

\begin{thm} \label{thm:main-cancellation}
Let $n \ge 0$ be even and let $E$, $E' \in \Proj(G,n)$. If $E \simeq E_0 \oplus \Z G$ and $E \oplus \Z G \simeq E' \oplus \Z G$, then $E \simeq E'$.
\end{thm}

\begin{proof}
Let $E_0 \in \hProj^{n+1}_{\Z G}(\Z,A_0)$, $E=(P_*,\partial_*) \in \hProj^{n+1}_{\Z G}(\Z,A)$ and $E' =(P_*',\partial_*') \in \hProj^{n+1}_{\Z G}(\Z,A')$. If $\psi: E \oplus \Z G \to E' \oplus \Z G$ denotes the given chain homotopy equivalence in $\hProj^{n+1}_{\Z G}(\Z,A_0 \oplus \Z G^2)$ and $\psi_A: A_0 \oplus \Z G^2 \to A' \oplus \Z G$ is the induced map on the left, consider $x = \psi_A^{-1}(0,1) \in \Um(A_0 \oplus \Z G^2)$.

We now claim that there exists a self chain homotopy equivalence $\varphi: E \oplus \Z G \to E \oplus \Z G$ such that the induced map $\varphi_A: A \oplus \Z G \to A \oplus \Z G$ has $\varphi_A(0,1) = x$. 

Let $F_1$, $F_2 \cong \Z G$ be defined so that $A = A_0 \oplus F_1$ and $A \oplus \Z G = A_0 \oplus F_1 \oplus F_2$. Since $\Proj^{n+1}_{\Z G}(\Z,A_0)$ is non-empty and $n+1$ is odd, we can combine Theorem \ref{thm:HK} and Lemma \ref{lemma:maschke} to get that there exists $\varphi_A \in \mathcal{G} = \langle E(F_1,A_0\oplus F_2), E(F_2,A_0\oplus F_1) \rangle \le GL(A_0 \oplus F_1 \oplus F_2)$ such that $\varphi_A(0,0,1)=x$. We claim that $\varphi_A$ can be extended to a chain homotopy equivalence $\varphi: E \oplus \Z G \to E \oplus \Z G$.

First recall that $E(F_2,A_0 \oplus F_1) = E(F_2,A) \le \GL(A \oplus F_2)$ is generated by elements of the form $\left(\begin{smallmatrix}  1 & 0 \\ f & 1\end{smallmatrix}\right)$ for $f: A \to F_2$ and $\left(\begin{smallmatrix}  1 & g \\ 0 & 1\end{smallmatrix}\right)$ for $g: F_2 \to A$.

If $i: A \hookrightarrow P$, then there exists $\widetilde{f} : P \to \Z G$ such that $\widetilde{f} \circ i = f$ by Lemma \ref{lemma:relative-inj}. It is straightforward to verify that the following diagrams commute, and so are chain homotopy equivalences.
\[
\begin{tikzcd}[row sep=.5cm, column sep=small]
E \oplus \Z G \ar[d,"\varphi_1"] \\
E \oplus \Z G
\end{tikzcd}
\hspace{-2.5mm} = \left( 
\begin{tikzcd}
[ampersand replacement=\&]
	0 \ar[r] \& A \oplus \Z G \ar[r,"\text{$\left(\begin{smallmatrix}  i & 0 \\ 0 & 1\end{smallmatrix}\right)$}"] \ar[d,"\text{$\left(\begin{smallmatrix}  1 & 0 \\ f & 1\end{smallmatrix}\right)$}"] \& P_n \oplus \Z G \ar[r,"\text{$(\partial_{n},0)$}"] \ar[d,"\text{$\left(\begin{smallmatrix}  1 & 0 \\ \widetilde{f} & 1\end{smallmatrix}\right)$}"] \& P_{n-1} \ar[r,"\partial_{n-1}"] \ar[d,"\id_{P_{n-1}}"] \& \cdots \ar[r,"\partial_1"] \& P_0 \ar[r] \ar[d,"\id_{P_0}"] \& 0 \\
	0 \ar[r] \& A \oplus \Z G \ar[r,"\text{$\left(\begin{smallmatrix}  i & 0 \\ 0 & 1 \end{smallmatrix}\right)$}"] \& P_n \oplus \Z G \ar[r,"\text{$(\partial_n,0)$}"] \& P_{n-1} \ar[r,"\partial_{n-1}"] \& \cdots \ar[r,"\partial_1"] \& P_0 \ar[r] \& 0
\end{tikzcd}
\right)
\]

\[
\begin{tikzcd}[row sep=.5cm, column sep=small]
E \oplus \Z G \ar[d,"\varphi_2"] \\
E \oplus \Z G
\end{tikzcd}
\hspace{-2.5mm} = \left( 
\begin{tikzcd}
[ampersand replacement=\&]
	0 \ar[r] \& A \oplus \Z G \ar[r,"\text{$\left(\begin{smallmatrix}  i & 0 \\ 0 & 1\end{smallmatrix}\right)$}"] \ar[d,"\text{$\left(\begin{smallmatrix}   1 & g \\ 0 & 1\end{smallmatrix}\right)$}"] \& P_n \oplus \Z G \ar[r,"\text{$(\partial_{n},0)$}"] \ar[d,"\text{$\left(\begin{smallmatrix}  1 & i \circ g \\ 0 & 1\end{smallmatrix}\right)$}"] \& P_{n-1} \ar[r,"\partial_{n-1}"] \ar[d,"\id_{P_{n-1}}"] \& \cdots \ar[r,"\partial_1"] \& P_0 \ar[r] \ar[d,"\id_{P_0}"] \& 0 \\
	0 \ar[r] \& A \oplus \Z G \ar[r,"\text{$\left(\begin{smallmatrix}  i & 0 \\ 0 & 1 \end{smallmatrix}\right)$}"] \& P_n \oplus \Z G \ar[r,"\text{$(\partial_n,0)$}"] \& P_{n-1} \ar[r,"\partial_{n-1}"] \& \cdots \ar[r,"\partial_1"] \& P_0 \ar[r] \& 0
\end{tikzcd}
\right)
\]

Similarly, we can show that the generators of $E(F_1,A_0 \oplus F_2)$ extend to chain homotopy equivalences. Hence, by writing $\varphi_A \in \mathcal{G}$ as the composition of maps of this form, we can get a chain homotopy equivalence $\varphi: E \oplus \Z G \to E \oplus \Z G$ by taking the composition of equivalences on each of the generators.

Now consider the map $\psi \circ \varphi = (\psi_A \circ \varphi_A,\psi_P \circ \varphi_P,\id, \cdots, \id): E \oplus \Z G \to E' \oplus \Z G$.
Since $(\psi_A \circ \varphi_A)(0,1)=(0,1)$, it must have the form $\psi_A \circ \varphi_A = \left(\begin{smallmatrix} \phi_A  & 0 \\ 0 & 1 \end{smallmatrix}\right)$ since it is an isomorphism. By commutativity, $(\psi_P \circ \varphi_P)(0,1) = (0,1)$ and so similarly $\psi_P \circ \varphi_P = \left(\begin{smallmatrix} \phi_P  & 0 \\ 0 & 1 \end{smallmatrix}\right)$ for some $\phi_P : P \to P'$. We are now done by noting that the triple $(\phi_A,\phi_P,\id, \cdots, \id)$ defines a chain homotopy equivalence $E \simeq E'$.
\end{proof}

We say that a graded tree is a \textit{fork} if it has a single vertex at each non-minimal grade and a finite set of a vertices at the minimal grade.

\begin{figure}[h] \vspace{-1.5mm}
\begin{tikzpicture}
\draw[fill=black] (0,0) circle (2pt);
\draw[fill=black] (1,0) circle (2pt);
\draw[fill=black] (2,0) circle (2pt);
\draw[fill=black] (3,0) circle (2pt);
\draw[fill=black] (4,0) circle (2pt);
\draw[fill=black] (2,1) circle (2pt);
\draw[fill=black] (2,2) circle (2pt);
\draw[fill=black] (2,3) circle (2pt);
\node at (2,3.6) {$\vdots$};
\draw[thick] (0,0) -- (2,1) (1,0) -- (2,1) (2,0) -- (2,1) (3,0) -- (2,1) (4,0) -- (2,1) -- (2,2) -- (2,3);
\end{tikzpicture}
\caption{A graded tree which is a fork} \vspace{-3mm}
\label{figure:fork}
\end{figure}

\begin{corollary} \label{cor:fork}
If $n \ge 0$ is even, $G$ is a finite group and $\chi \in C(\Z G)$, then $\Proj(G,n;\chi)$ is a fork. In particular, $\Alg(G,n)$ is a fork for $n \ge 2$ even.
\end{corollary}

This recovers the even-dimensional case of a result of Browning \cite[Theorem 5.4]{Br79a}. This fails in odd dimensions, i.e. there are examples of finite groups $G$ for which $\Alg(G,n)$ is not a fork for some $n$ odd \cite{Dy79}.

\subsection{Projective $0$-complexes and the unstable Euler class}

We now consider the case $n=0$. Recall that $P(\Z G)$ denotes the set of $\Z G$-module isomorphism classes of (finitely generated) non-zero projective $\Z G$-modules. 
This is a graded graph with edges between each $P$ and $P \oplus \Z G$.

Note that a projective $0$-complex has the form
\[ E = (0 \to A \xrightarrow[]{i} P \xrightarrow[]{\varepsilon} \Z \to 0 ),\]
and so consists of a non-zero projective module $P \in P(\Z G)$ as well as the additional data $(A, i, \varepsilon)$.
If $\wh e$ is the unstable Euler class, then $\wh e : \Proj(G,0) \to P(\Z G)$ is a map of graded graphs since $\wh e(E \oplus \Z G) \cong  \wh e(E) \oplus \Z G$. We will now show the following:

\begin{thm} \label{thm:stab1-to-modules}
If $G$ is a finite group, then the unstable Euler class gives an isomorphism of graded graphs
		\[ \wh e : \Proj(G,0) \to P(\Z G). \]
\end{thm}

\begin{remark}
Such a statement is implicit in the proof of \cite[Theorem IV, Theorem 57.4]{Jo03a}, though the argument there contains an error and can only be used to recover the statement above in the case of projective modules of rank one. This, however, suffices since one can instead rely on the cancellation theorems of Hambleton-Kreck \cite[Theorem B]{HK93} or Browning \cite[Theorem 5.4]{Br79a} at that stage in the proof.
\end{remark}

\begin{proof}
To see that $\wh e$ is surjective, let $P \in P(\Z G)$. By Corollary \ref{cor:proj-surj}, there is a surjection $\varphi: P \to \Z$ and this defines an extension $E = (P,-) \in \hProj_{\Z G}^{1}(\Z,\Ker(\varphi))$ which has $\wh e(E) = P$.

We will now show injectivity. First let $E = (P,-) \in \hProj_{\Z G}^1(\Z,A)$ and let $E' = (P,-) \in \hProj_{\Z G}^1(\Z,A')$. We will begin by considering the case where $P$ has rank one.
To show that $E \simeq E'$, it suffices to find isomorphisms $\varphi_A: A \to A'$ and $\varphi_{\Z}: \Z \to \Z$ such that the following diagram commutes:
\[
\begin{tikzcd}[row sep=.5cm, column sep=small]
E \ar[d,"\varphi"] \\
E'
\end{tikzcd}
= \left(
\begin{tikzcd}[row sep=.5cm]
	0 \ar[r] & A \ar[r,"i"] \ar[d,dashed,"\varphi_A"] & P \ar[r,"\varepsilon"] \ar[d,"\id"] & \Z \ar[r] \ar[d,dashed,"\varphi_{\Z}"] & 0 \\
	0 \ar[r] & A' \ar[r,"i'"] & P \ar[r,"\varepsilon'"]& \Z \ar[r] & 0
\end{tikzcd}
\right).
\]

Consider the maps $\bar{\varepsilon} = \varepsilon \otimes \Q$, $\bar{\varepsilon}' = \varepsilon' \otimes \Q: P \otimes \Q \cong \Q G \to \Q$. Since $\Q$ has trivial $G$-action, each map is determined by the fact that $\bar{\varepsilon}(g)=\bar{\varepsilon}'(g) = 0$ for all $g \in G$ and $\bar{\varepsilon}(1) = \bar{\varepsilon}'(1) = x_i$ for some $x_i \in \Q^\times$. 
Hence $\Ker(\bar{\varepsilon}) = \Ker(\bar{\varepsilon}')$ and so $(\varepsilon' \circ i) \otimes \Q = 0$. 
Since $A$ is a $\Z G$ lattice, this implies that $\varepsilon' \circ i = 0$ and so we can define maps $\varphi_A$ and $\varphi_{\Z}$ as above. Now $\varphi_{\Z}$ is necessarily surjective and so an isomorphism. Hence $\varphi_A$ is an isomorphism by the five lemma, and so $E \simeq E'$.

Now suppose $E$, $E'$ as above but with $\rank(P) \ge 2$. By Proposition \ref{prop:swan-induced}, this implies that there exists $P_0$ of rank one such that $P \cong P_0 \oplus \Z G^i$ for some $i \ge 1$. Since $\wh e$ is surjective, there exists $E_0 = (P_0,-) \in \hProj^1_{\Z G}(\Z,A_0)$ for some $A_0$. By Theorem \ref{thm:proj-stability}, there exists $j \ge 0$ for which $E_0 \oplus \Z G^{i+j} \simeq E \oplus \Z G^j \simeq E' \oplus \Z G^j$. Since $i \ge 1$, Theorem \ref{thm:main-cancellation} then implies that $E_0 \oplus \Z G^i \simeq E \simeq E'$.
\end{proof}

For use in later sections, it will be necessary to further refine the isomorphism given by $\wh e$. Consider following two decompositions (where $\cong$ denotes bijection):
\[
\Proj(G,0) \cong \coprod_{\chi \in C(\Z G)} \Proj(G,0;\chi) \cong \coprod_{A \in \Mod(\Z G)} \hProj_{\Z G}^{1}(\Z,A).
\]

We will begin by determining the image of $\Proj(G,0;\chi)$ under $\wh e$.
This is immediate from Theorem \ref{thm:stab1-to-modules} and the definition of $\hProj_{\Z G}^{1}(\Z,A;\chi)$. For convenience, we will write $\chi = [P]$ for some $P \in P(\Z G)$.

\begin{prop} \label{prop:stab1-to-modules}
Let $P \in P(\Z G)$. Then there is an isomorphism of graded trees
		\[ \wh e : \Proj(G,0;[P]) \to [P].\]
\end{prop}

We will next determine the image of $\Proj_{\Z G}^1(\Z,A)$ under $\wh e$ for $A$ a fixed $\Z G$-module such that $\Proj_{\Z G}^1(\Z,A)$ is non-empty. Recall that, if $E \in \Proj^1_{\Z G}(\Z,A)$, then Proposition \ref{prop:congruence-classification} implies that there is a bijection
\[ (m_{\cdot})^*: (\Z /|G|)^\times \to \Proj^1_{\Z G}(\Z,A)\]
given by $r \mapsto (m_r)^*(E)$, where $m_r: \Z \to \Z$ denotes multiplication by $r$.

If $M$ is a (left) $\Z G$-module and $r \in (\Z/|G|)^\times$, then the tensor product $(I,r) \otimes M$ can be considered as a (left) $\Z G$-module since $(I,r)$ is a two-sided ideal. This allows us to find an explicit form for pullbacks of extensions. We will begin with the following special case.

\begin{lemma} \label{lemma:pullback=swan-module-action}
Let $A$ be a $\Z G$-module and suppose $E = (P,-) \in \Proj_{\Z G}^1(\Z,A)$ where $\rank(P)=1$. Then, for any $r \in (\Z/|G|)^\times$, there are maps $\bar{i}$, $\bar{\varepsilon}$ such that
\[ (m_r)^*(E) \cong ( 0 \to A \xrightarrow[]{\bar{i}} (I,r) \otimes P \xrightarrow[]{\bar{\varepsilon}}  \Z \to 0).\]
\end{lemma}

\begin{proof}
Let $E = (P,-) \in \Proj_{\Z G}^1(\Z,A)$ and note that we have the following diagrams
\[
\begin{tikzcd}
(I,r) \ar[r,"\frac{1}{r} \varepsilon"] \ar[d,"i"] & \Z \ar[d,"r"] \\
\Z G \ar[r,"\varepsilon"] & \Z	
\end{tikzcd}
\qquad
\begin{tikzcd}
(I,r) \otimes P \ar[r,"\frac{1}{r} \varepsilon \otimes 1"] \ar[d,"i \otimes 1"] & \Z \otimes P \ar[d,"r \otimes 1"] \\
\Z G \otimes P \ar[r,"\varepsilon \otimes 1"] & \Z	 \otimes P
\end{tikzcd}
\]
where $i : (I,r) \hookrightarrow \Z G$ is inclusion. It can be checked directly that the first diagram is a pullback, and this implies that the second diagram is a pullback since $P$ is projective and so flat.
Since $\rank(P)=1$, we can choose identifications $\Z G \otimes P \cong P$ and $\Z \otimes P \cong \Z$ for which $\varepsilon \otimes 1$ corresponds to $\varepsilon^E$. We now have a map 
$(\id_{A}, \varphi, m_r): E' \to E$ 
where $E' = ((I,r) \otimes P,-)$. Hence $E' \cong (m_r)^*(E)$ by uniqueness of pullbacks.
\end{proof}

We can now upgrade this to the general case using Theorem \ref{thm:stab1-to-modules}.

\begin{lemma} \label{lemma:image-of-unstable}
Let $A$ be a $\Z G$-module and suppose $E = (P,-) \in \Proj^1_{\Z G}(\Z,A)$.
\begin{enumerate}[\normalfont (i)]
\item
There exists a projective $\Z G$-module $P_0$ with $\rank(P_0)=1$ and $k \ge 0$ such that $P \cong P_0 \oplus \Z G^k$ and
\[ E \cong (0 \to A \xrightarrow[]{i_0} P_0 \oplus \Z G^k \xrightarrow[]{(\varepsilon_0,0)} \Z \to 0)\]
for some maps $i_0$ and $\varepsilon_0 : P_0 \to \Z$.
\item
With $P_0$, $i_0$ and $\varepsilon_0$ as above, we have that
\[ (m_r)^*(E) \cong ( 0 \to A \xrightarrow[]{\bar{i}_0} ((I,r) \otimes P_0) \oplus \Z G^k \xrightarrow[]{(\bar{\varepsilon}_0,0)}  \Z \to 0)\]
for some maps $\bar{i}_0$ and $\bar{\varepsilon}_0 : (I,r) \otimes P_0 \to \Z$.
\end{enumerate}
\end{lemma}

\begin{proof}
(i) Since $P \twoheadrightarrow \Z$, we know that $P$ is non-zero. Hence, by Proposition \ref{prop:swan-induced}, there exists a projective $\Z G$-module $P_0$ with $\rank(P_0)=1$ and $k \ge 0$ such that $P \cong P_0 \oplus \Z G^k$. Since $\wh e$ is an isomorphism of graded trees, there exists $E_0 \in \Proj_{\Z G}^1(\Z,A_0)$ for some $\Z G$-module $A_0$ such that $E \cong E_0 \oplus \Z G^k$. Write
\[ E_0 = (0 \to A_0 \xrightarrow[]{i_0'} P_0 \xrightarrow[]{\varepsilon_0} \Z \to 0)\]
for some $i_0'$ and $\varepsilon_0$. The result follows by forming $E_0 \oplus \Z G^k$.

(ii) The result follows by noting that $(m_r)^*(E_0 \oplus \Z G^k) \cong (m_r)^*(E_0) \oplus \Z G^k$ and evaluating $(m_r)^*(E_0)$ using Lemma \ref{lemma:pullback=swan-module-action}.
\end{proof}

\begin{remark}
The proof of (i) also implies that $A \cong A_0 \oplus \Z G^k$.
\end{remark}

This implies the following. This is the analogue of Proposition \ref{prop:image-of-chi} which established the corresponding result for the stable Euler class $e$.

\begin{prop} \label{prop:image-of-unstable}
Let $A$ be a $\Z G$-module and suppose $E = (P,-) \in \Proj^1_{\Z G}(\Z,A)$. Then we have
\[ \wh e( \Proj^1_{\Z G}(\Z,A)) = \{ ((I,r) \otimes P_0) \oplus \Z G^k : r \in (\Z/|G|)^\times\} \subseteq P(\Z G)\]
where $P_0$ is any rank one projective $\Z G$-module such that $P \cong P_0 \oplus \Z G^k$ for $k \ge 0$.
\end{prop}

For completeness, as well as for later use, we will note the following which is a consequence of \cite[Remark 1.30]{FRU74}. This shows that Propositions \ref{prop:image-of-chi} and \ref{prop:image-of-unstable} agree in the case $n=1$.

\begin{lemma} \label{lemma:tensoring-with-swan-modules}
Let $P$ be a projective $\Z G$-module with $\rank(P)=1$ and let $r \in (\Z/|G|)^\times$. Then
\[ [(I,r) \otimes P] = [(I,r)] + [P] \in C(\Z G).\]
\end{lemma}


\section{Polarised homotopy classification of $(G,n)$-complexes}
\label{section:classification-of-Alg(G,n)}

Recall that, for a group $G$, a \textit{$G$-polarised} space is a pair $(X,\rho_X)$ where $X$ is a topological space and $\rho_X : \pi_1(X,*) \to G$ is a given isomorphism. We say that two $G$-polarised spaces $(X,\rho_X)$, $(Y,\rho_Y)$ are \textit{polarised homotopy equivalent} if there exists a homotopy equivalence $h: X \to Y$ such that $\rho_X = \rho_Y \circ \pi_1(h)$.

Let $\PHT(G,n)$ denote the set of polarised homotopy types of finite $(G,n)$-complexes over $G$. This is a graded graph with edges between each $(X,\rho_X)$ and $(X \vee S^2, (\rho_X)^+)$ where $(\rho_X)^+$ is induced by $\rho_X$ and the collapse map $X \vee S^2 \to X$.

If $X$ is a finite CW-complex, then the cellular chain complex $C_*(\widetilde{X})$ can be viewed as a chain complex of $\Z[\pi_1(X)]$-modules under the monodromy action. We can use a $G$-polarisation $\rho : \pi_1(X) \to G$ to get a chain complex of $\Z G$-modules $C_*(\widetilde{X}, \rho)$ which is the same as $C_*(\widetilde{X})$ as a chain complex of abelian groups but with action $g \cdot x = \rho^{-1}(g)x$ for all $g \in G$ and $x \in C_i(\widetilde{X})$ for some $i \ge 0$.

The following is a mild generalisation of \cite[Theorem 1.1]{Ni19}:

\begin{prop} \label{prop:pht-to-alg}
Let $G$ be a finitely presented group and let $n \ge 2$. Then there is an injective map of graded trees
\[ \widetilde{C}_*: \PHT(G,n) \to \Alg(G,n)\]
induced by the map $(X,\rho) \mapsto C_*(\widetilde{X},\rho)$. Furthermore:
\begin{clist}{(i)}
\item
If $n \ge 3$, then $\widetilde{C}_*$ is bijective. 
\item
If $n=2$, then $\widetilde{C}_*$ is bijective if and only if $G$ has the {\normalfont D2} property.
\end{clist}
\end{prop}

\begin{remark} (a) Even if $G$ does not satisfy the D2 property, Proposition \ref{prop:pht-to-alg} can be replaced with an isomorphism
$ \widetilde{C}_*: \D(G) \to \Alg(G,2)$
where $\D(G)$ denotes the polarised homotopy tree of D2-complexes over $G$ \cite[Theorems 1.1]{Ni19}. 

\noindent (b) This is often vacuous in the case $n \ge 3$ since $\PHT(G,n)$ and $\Alg(G,n)$ are often empty. More specifically, $\PHT(G,n)$ is non-empty if and only if $G$ is of type $\FF_n$. $\Alg(G,n)$ is non-empty if and only if $G$ has type $\FP_n$ (see \cite{Bi81}), and it is well-known that $\FF_n \Leftrightarrow \FP_n$ for finitely presented groups. 
This situation arises since there exists finitely presented groups which are not of type $\FF_n$ for $n \ge 3$ \cite{St63}.

\noindent (c) This fails in general for non-finitely presented groups. In particular, for each $n \ge 2$, Bestvina-Brady constructed a non-finitely presented group $G$ of type $\FP_n$ \cite{BB97}. Here $\PHT(G,n)$ is empty and $\Alg(G,n)$ is non-empty and so $\wt C_*$ is not bijective.
\end{remark}

We will now use the results from the previous section to study projective $n$-complexes over groups with periodic cohomology. By Proposition \ref{prop:pht-to-alg}, this will lead to a proof of the following more detailed version of Theorem \ref{thm:main-topological-I}. Note that, if $X$ is a finite $(G,n)$-complex, then
\[\pi_n(X) \cong H_n(\wt X) \cong \Ker(\partial_n: C_n(\wt X) \to C_{n-1}(\wt X))\]
are isomorphisms of $\Z G$-modules. 

\begin{thm} \label{thm:main-topological-I-detailed}
Let $G$ have $k$-periodic cohomology, let $n=ik$ or $ik-2$ for some $i \ge 1$ and let $P_{(G,n)}$ be a projective $\Z G$-module with $\sigma_{ik}(G) = [P_{(G,n)}] \in C(\Z G)/T_G$.
Let $F \in \Proj^{ik}_{\Z G}(\Z,\Z)$ be such that $e(F) = [P_{(G,n)}]$. Then there is an injective map of graded trees
\[ \Psi: \PHT(G,n) \to [P_{(G,n)}]\] 
which is defined as follows.

\begin{clist}{(i)}
\item If $n = ik-2$, then 
 $\Psi : X \mapsto P$, where $P$ is the unique projective $\Z G$-module for which \\
$(0 \to \Z \xrightarrow[]{\alpha} P^* \xrightarrow[]{\beta} \pi_n(X) \to 0) \circ C_*(\wt X) \simeq F$ for some $\alpha$, $\beta$.

\item If $n = ik$, then $\Psi : X \mapsto P$, where $P$ is the unique projective $\Z G$-module for which \\ $C_*(\wt X) \simeq (0 \to \pi_n(X) \xrightarrow[]{\alpha} P \xrightarrow[]{\beta} \Z \to 0) \circ F$ for some $\alpha$, $\beta$.
\end{clist}
Furthermore, $\Psi$ is bijective if and only if $n \ge 3$ or $n=2$ and $G$ has the {\normalfont D2} property.
\end{thm}

\begin{remark} \label{remark:k=2}
The definition of $P_{(G,n)}$ depends on $G$, $n$ and $k$. 
Note that $n$ and $k$ determine $i$ except when $k=2$ where $n=ik=(i+1)k-2$.
However, in this case there is no ambiguity since $G$ is cyclic \cite[Lemma 5.2]{Sw65} and so $\sigma_{2i}(G) = 0$ for all $i$.
\end{remark}

First note that, when $G$ has periodic cohomology, we get the following two relations between projective complexes of different dimensions.

\begin{lemma} \label{lemma:proj-complex-relations}
Suppose $G$ has $k$-periodic cohomology and let $F \in \Proj^k_{\Z G}(\Z,\Z)$. If $n \ge 0$, then we have the following isomorphisms of graded graphs
\begin{align*} - \circ F : & \, \Proj(G,n) \to \Proj(G,n+k)\\
 * \circ \Psi_F : & \, \Proj(G,n) \to \Proj(G,k-(n+2))\end{align*}
 where $n+2 \le k$ is the second case.
\end{lemma}

\begin{proof}
The first isomorphism is immediate from the shifting lemma. The second isomorphism consists of the compositions
\[ \hProj_{\Z G}^{n+1} (\Z, A) 
\xrightarrow[]{\Psi_F} \hProj_{\Z G}^{k-n-1}(A,\Z)
\xrightarrow[]{*} \hProj_{\Z G}^{k-n-1}(\Z,A^*)\]
for all $\Z G$-modules $A$. These are bijections by the duality and reflexivity lemmas.

To see that the image of the full map is $\Proj(G,k-(n+2))$ note that, if $B$ is such that $\hProj_{\Z G}^{k-n-2} (\Z, B)$ is non-zero, then $B$ is a $\Z G$-lattice since it is a submodule of a free module. By Lemma \ref{lemma:lattice**=lattice}, we have that $B^{**} \cong B$ and so there is an isomorphism $\ast \circ \Psi_F : \hProj_{\Z G}^{n+1} (\Z, B^*) \to \hProj_{\Z G}^{k-n-1} (\Z, B)$.
\end{proof}

\begin{remark} \label{remark:proj-complex-relations} Furthermore, if $E \in \Proj(G,n)$ has $\chi = e(E)$, then it is easy to see that $e(E \circ F) = e(F) + \chi$ since $k$ is even and $e((\psi_F(E))^*) = e(F)^*-\chi^*$. \end{remark}

The proof of Theorem \ref{thm:main-topological-I-detailed} will now consist of applying Lemma \ref{lemma:proj-complex-relations} in the case $k \mid n$ or $n+2$ and then composing with the isomorphism from Theorem \ref{thm:stab1-to-modules}.

We will need the following result of Wall \cite[Corollary 12.6]{Wa79a}.

\begin{prop} \label{prop:obstruction=2-torsion}
If $G$ has $k$-periodic cohomology, then 
\[2 \sigma_k(G) = 0 \in C(\Z G)/T_G.\]	
\end{prop}

By iterating extensions using the Yoneda product, it can be shown that $n \sigma_k(G) = \sigma_{nk}(G)$ and so this theorem is equivalent to showing that $\sigma_{2k}(G)=0$, i.e. that the obstruction vanishes whenever $k$ is not the minimal period.	

\begin{thm} \label{thm:case-n=k}
Suppose $G$ has $k$-periodic cohomology and $\sigma_k(G) = [P_{(G,n)}]+T_G$ for some $P_{(G,n)} \in P(\Z G)$. Then there exists $F \in \Proj^k_{\Z G}(\Z,\Z)$ such that there are isomorphisms of graded trees
\[\Alg(G,k) \xrightarrow[]{(- \circ F)^{-1}} \Proj(G,0;[P_{(G,n)}]) \xrightarrow[]{\quad \wh e \quad } [P_{(G,n)}].\]
\end{thm}

\begin{proof}
By Proposition \ref{prop:obstruction=2-torsion}, we have that $\sigma_k(G) = [P_{(G,n)}]+T_G = -[P_{(G,n)}]+T_G$ and so there exists $F \in \Proj^k_{\Z G}(\Z,\Z)$ with $e(F) = -[P_{(G,n)}]$ by Proposition \ref{prop:image-of-chi}.

If $E \in \Alg(G,k)$, then $e(E) = 0$ and so $e((-\circ F)^{-1}) = -(-1)^k e(F)$ by Lemma \ref{lemma:e-of-prod}. Since $k$ is even, this is equal to $[P_{(G,n)}]$. Hence the map $(- \circ F)^{-1}$ is as described.
By Lemma \ref{lemma:proj-complex-relations}, we get that  $ (- \circ F)^{-1}$ is an isomorphism. 

That $\wh e$ is an isomorphism follows from Proposition \ref{prop:stab1-to-modules}.	
\end{proof}

\begin{thm} \label{thm:case-n=k+2}
Suppose $G$ has $k$-periodic cohomology and $\sigma_k(G)=[P_{(G,n)}]+T_G$ for some $P_{(G,n)} \in P(\Z G)$. Then there exists $F \in \Proj^k_{\Z G}(\Z,\Z)$ such that there are isomorphisms of graded trees
\[ \Alg(G,k-2) \xrightarrow[]{* \circ \Psi_F} \Proj(G,0;[P_{(G,n)}]) \xrightarrow[]{\hspace{2.5mm} \wh e \hspace{2.5mm}} [P_{(G,n)}].\]
\end{thm}

\begin{proof}
By Propositions \ref{prop:[P]=-[P^*]-mod-T_G}, we have that $\sigma_k(G) = [P_{(G,n)}]+T_G = -[P_{(G,n)}^*]+T_G$ and so there exists $F \in \Proj^k_{\Z G}(\Z,\Z)$ with $e(F) = -[P_{(G,n)}^*]$ by Proposition \ref{prop:image-of-chi}.

If $E \in \Alg(G,k-2)$, then $e(\Psi_F(E)) = -e(F)$ by Lemma \ref{lemma:e-of-prod}. This implies that 
\[ e((\ast \circ \Psi_F)(E)) = -e(F)^* = [P_{(G,n)}]\]
and so the map $\ast \circ \Psi_F$ is as described. By Lemma \ref{lemma:proj-complex-relations}, $\ast \circ \Psi_F$ is an isomorphism.

That $\wh e$ is an isomorphism follows from Proposition \ref{prop:stab1-to-modules}, as in the previous theorem.
\end{proof}

\begin{proof}[Proof of Theorem \ref{thm:main-topological-I-detailed}]
If $G$ has $k$-periodic cohomology, then it also has $ik$-periodic cohomology for any $i \ge 1$. Hence, by swapping $k$ for $ik$, we can assume $i=1$. By combining Theorems \ref{thm:case-n=k} and \ref{thm:case-n=k+2} with Proposition \ref{prop:pht-to-alg}, we obtain injective maps of graded trees $\Psi : \PHT(G,n) \to [P_{(G,n)}]$ for $n = k$ or $k-2$, which are bijective as required.
It remains to show that, in each case, $\Psi$ has the form given in (i), (ii).
	
If $n = k-2$, then $(\ast \circ \Psi_F)(C_*(\wt X)) \simeq (0 \to A \to P \to \Z \to 0)$ for some $A$ and some $P \in [P_{(G,n)}]$. By Lemma \ref{lemma:double-dual}, we have $\Psi_F(C_*(\wt X)) \simeq (0 \to \Z \to P^* \to A^* \to 0)$. Hence $A^* \cong \pi_n(X)$ and $(0 \to \Z \to P^* \to A^* \to 0) \circ C_*(\wt X) \simeq F$.

If $n=k$, then $(- \circ F)^{-1}(C_*(\wt X)) \simeq (0 \to A \to P \to \Z \to 0)$ for some $A$ and some $P \in [P_{(G,n)}]$. Hence $C_*(\wt X) \simeq (0 \to A \to P \to \Z \to 0) \circ F$ and $A \cong \pi_n(X)$.
\end{proof}

This completes the proof of Theorem \ref{thm:main-topological-I}.


\section{Homotopy classification of $(G,n)$-complexes} 
\label{section:homotopy-classification}

For a finitely presented group $G$, an automorphism $\theta \in \Aut(G)$ acts on $\PHT(G,n)$ by sending $(X,\rho) \mapsto (X,\theta \circ \rho)$. It is straightforward to see that
\[ \HT(G,n) \cong \PHT(G,n)/ \Aut(G)\]
and the goal of this chapter will be to determine the induced action of $\Aut(G)$ on $[P_{(G,n)}]$ under the isomorphism $\PHT(G,n) \cong [P_{(G,n)}]$ obtained in Theorem \ref{thm:main-topological-I-detailed}.

\subsection{Preliminaries on the action of $\Aut(G)$}

We begin by defining natural actions of $\Aut(G)$ on $\Z G$-modules and chain complexes of $\Z G$-modules.
Firstly, for a $\Z G$-module $A$ and $\theta \in \Aut(G)$, let $A_\theta$ denote the $\Z G$-module whose underlying abelian group is that of $A$ and whose action is $g \cdot x = \theta(g)x$ where $g \in G$, $x \in A$. This action has the following basic properties:

\begin{lemma} \label{lemma:theta-basic-facts}
Let $\theta \in \Aut(G)$. Then 
\begin{enumerate}[\normalfont (i)]
\item
There is a $\Z G$-module isomorphism
\[ i_\theta :   \Z G \to \Z G_\theta, \quad  \sum_{g \in G} a_i g_i \mapsto \sum_{g \in G} a_i \theta(g_i).\]		
\item
If $A, B \in \Mod(\Z G)$, then $(A \oplus B)_\theta \cong A_\theta \oplus B_\theta$.
\item
If $P \in P(\Z G)$, then $P_\theta \in P(\Z G)$.
\end{enumerate}	
\end{lemma}

We can extend the action to chain complexes as follows. If $A, B$ are $\Z G$-modules and $E = (E_*,\partial_*) \in \Ext_{\Z G}^n(A,B)$, then we define $E_\theta \in \Ext_{\Z G}^n(A_\theta,B_\theta)$ by
\[ E_\theta = (  0 \to B_\theta \xrightarrow[]{\partial_{n}} (E_{n-1})_\theta \xrightarrow[]{\partial_{n-1}} (E_{n-2})_\theta \to \cdots \to (E_{1})_\theta \xrightarrow[]{\partial_{1}} (E_{0})_\theta \xrightarrow[]{\partial_{0}} A_\theta \to 0). \]
It is easy to see that this is well-defined up to chain homotopy and, by the lemma above, it preserves projective extensions and so also induces a map on $\hProj^n_{\Z G}(A,B)$.
This following is immediate from the definition of $\widetilde{C}_*(X,\rho)$.

\begin{lemma} \label{lemma:action-on-alg}
If $E \in \Alg(G,n)$, then the induced action of $\theta \in \Aut(G)$ on $E$ is given by $\theta \cdot E = E_\theta$, i.e. if $E = \widetilde{C}_*(X,\rho)$, then $E_\theta = \widetilde{C}_*(X,\theta \circ \rho)$.
\end{lemma}

We now establish a few basic properties of this action which we will use later in this section. From now on, we will specialise to the case where $G$ is a finite group. Firstly, we note that the action commutes with dualising.

\begin{lemma}
If $A$ and $B$ are $\Z G$-lattices, $E \in \Proj^n_{\Z G}(A,B)$ for $n \ge 1$ and $\theta \in \Aut(G)$, then \[(E_\theta)^* \cong (E^*)_\theta.\]
\end{lemma}

\begin{proof}
We begin by proving the corresponding statement for modules, i.e. that, if $A$ is a $\Z G$-lattice, then $(A_\theta)^* \cong (A^*)_\theta$. Let $A \cong_{\text{Ab}} \Z^k$, so that the $\Z G$-module structure is determined by an integral representation $\rho_A: G \to \GL_k(\Z)$. As remarked earlier, we have that $\rho_{A^*}(g) = \rho_A(g^{-1})^T$ and it is easy to see that $\rho_{A_\theta} = \rho_A \circ \theta$. Therefore $(A_\theta)^* \cong (A^*)_\theta$ follows by noting that
\[ \rho_{(A_\theta)^*}(g) = \rho_{A_\theta}(g^{-1})^T = \rho_A(\theta(g^{-1}))^T = \rho_A(\theta(g)^{-1})^T\]
and
\[ \rho_{(A^*)_\theta}(g) = \rho_{A^*}(\theta(g)) = \rho_A(\theta(g)^{-1})^T.\]
The result for extensions now follows immediately since $\theta$ only affects the underlying modules and not the maps between them.	
\end{proof}

In light of this, for $\Z G$-lattices $A$ and $B$ and $E \in \Proj^n_{\Z G}(A,B)$, it now makes sense to write $A_\theta^*$ and $E_\theta^*$. Note that the action also commutes with pushouts.

\begin{lemma} \label{lemma:theta-commutes-with-pullbacks}
If $\theta \in \Aut(G)$, $f: B_1 \to B_2$ is a $\Z G$-module homomorphism and $E \in \Ext^n_{\Z G}(A,B_1)$, then
\[ f_*(E_\theta) \cong (f_*(E))_\theta.\]	
\end{lemma}

\subsection{Proof of Theorem \ref{thm:main-topological-II}}

In the case where $A=B=\Z$, we can consider this as an action on $\Proj^n_{\Z G}(\Z,\Z)$ by using the identification $\Z_\theta \cong \Z$. 

\begin{lemma} \label{lemma:theta-action=pullback}
If $G$ has $k$-periodic cohomology, then there exists a unique map $\psi_k : \Aut(G) \to (\Z /|G|)^\times$ such that, for every $F \in \Proj^k_{\Z G}(\Z,\Z)$ and $\theta \in \Aut(G)$, we have
\[ F_\theta \cong (m_{\psi_k(\theta)})_*(F).\]	
\end{lemma}

\begin{proof}
Fix an extension $F_0 \in \Proj^k_{\Z G}(\Z,\Z)$. By dualising and then applying Proposition \ref{prop:congruence-classification}, it follows that every extension in $\Proj^k_{\Z G}(\Z,\Z)$ is of the form $(m_r)_*(F_0)$ for some $r \in (\Z/|G|)^\times$. For $\theta \in \Aut(G)$, define $\psi_k(\theta) =r \in (\Z/|G|)^\times$ for any $r \in (\Z/|G|)^\times$ such that $(F_0)_\theta \cong (m_r)_*(F_0)$. 

If $F \in \Proj^k_{\Z G}(\Z,\Z)$, then $F \cong (m_r)_*(F_0)$ for a unique $r \in (\Z/|G|)^\times$. By Lemma \ref{lemma:theta-commutes-with-pullbacks}, we now have that
\begin{align*} 
F_\theta 
&\cong ((m_r)_*(F_0))_\theta 
\cong (m_r)_*((F_0)_\theta) 
\cong (m_r)_*((m_{\psi_n(\theta)})_*(F_0)) \\
&\cong (m_{\psi_n(\theta)})_*((m_r)_*(F_0)) 
\cong (m_{\psi_n(\theta)})_*(F).\qedhere
\end{align*}
\end{proof}

\begin{lemma} \label{lemma:iterated-extensions-pullback}
	If $E$, $E' \in \Proj^k_{\Z G}(\Z,\Z)$ and $r \in \Z$ coprime to $|G|$, then
	\[ E \circ (m_r)_*(E') \cong (m_r)_*(E) \circ E'.\]
\end{lemma}

\begin{proof}
Consider the pushout map $\nu: E' \to (m_r)_*(E')$. Since this induces $m_r$ on the left copy of $\Z$, we can extend it to a map $\widetilde{\nu}: E \circ E' \to E \circ (m_r)_*(E')$ which induces multiplication by $r \in \Z \subseteq \Z G$ on every module in $E$, i.e.
\[
\hspace{-2mm}
\begin{tikzcd}
E \circ E' \ar[d,"\widetilde{\nu}"] \\
E \circ (m_r)_*(E')
\end{tikzcd}
\hspace{-6mm} = \left(
\begin{tikzcd}[column sep=.5cm]
0 \ar[r] & \Z \ar[r,"i"] \ar[d,"r"] & P_{k-1} \ar[r,"\partial_{k-1}"] \ar[d,"r"] & \cdots \ar[r,"\partial_1"] & P_0 \ar[r,"i' \circ \varepsilon"] \ar[d,"r"] & P_{k-1}' \ar[r,"\partial_{k-1}'"] \ar[d,"\nu_{k-1}"] & \cdots \ar[r,"\partial_1'"] & P_0' \ar[r,"\varepsilon'"] \ar[d,"\nu_0"] & \Z \ar[r] \ar[d,"1"] & 0\\
0 \ar[r] & \Z \ar[r,"i"] & P_{k-1} \ar[r,"\partial_{k-1}"] & \cdots \ar[r,"\partial_1"] & P_0 \ar[r,"i' \circ \varepsilon"] & P_{k-1}' \ar[r,"\partial_{k-1}'"] & \cdots \ar[r,"\partial_1'"] & P_0' \ar[r,"\varepsilon'"] & \Z \ar[r] & 0
\end{tikzcd}
\right).
\]
By the uniqueness of pushouts, this implies that $E \circ (m_r)_*(E') \cong (m_r)_*(E \circ E') = (m_r)_*(E) \circ E'$ as required.	
\end{proof}

Note that, if $G$ has $k$-periodic cohomology and $k \mid n$, then it also has $n$-periodic cohomology and so $\psi_n$ can still be defined using Lemma \ref{lemma:theta-action=pullback}.
The above lemma now allows us to give the following relation between $\psi_k$ and $\psi_n$ for $k \mid n$.

\begin{lemma} \label{lemma:theta-action-powers}
If $G$ has $k$-periodic cohomology, $i \ge 1$ and $\theta \in \Aut(G)$, then \[\psi_{ik}(\theta) = \psi_k(\theta)^i.\]
\end{lemma}

\begin{proof}
Let $F \in \Proj^k_{\Z G}(\Z,\Z)$ and consider $F^i \in \Proj^{ik}_{\Z G}(\Z,\Z)$. Then Lemma \ref{lemma:theta-action=pullback} implies that $F_\theta \cong (m_{\psi_k(\theta)})_*(F)$ and $(F^i)_\theta \cong (m_{\psi_{ik}(\theta)})_*(F^i)$. Since $(F^i)_\theta \cong (F_\theta)^i$, this implies that $(m_{\psi_{ik}(\theta)})_*(F^i) \cong ((m_{\psi_{k}(\theta)})_*(F))^i$.

By repeated application of Lemma \ref{lemma:iterated-extensions-pullback}, we get that 
\[(m_{\psi_{ik}(\theta)})_*(F^i) \cong ((m_{\psi_{k}(\theta)})_*(F))^i \cong (m_{\psi_{k}(\theta)})_*^i(F^i) \cong (m_{\psi_{k}(\theta)^i})_*(F^i)\]
and so $\psi_{ik}(\theta) \cong \psi_k(\theta)^i$ mod $|G|$ by the extension of Proposition \ref{prop:congruence-classification} to arbitrary extensions via the shifting lemma.	
\end{proof}

In order to prove Theorem \ref{thm:main-topological-II}, it suffices to check what the action of $\Aut(G)$ corresponds to under the isomorphisms described in Section \ref{section:classification-of-Alg(G,n)}. Similarly to Section \ref{section:classification-of-Alg(G,n)}, it will suffice to consider the cases where $k=n$ or $n+2$.

\begin{thm} \label{thm:action-k|n}
Suppose $G$ has $k$-periodic cohomology and $\sigma_k(G) = [P_{(G,n)}]+T_G$ for some $P_{(G,n)} \in P(\Z G)$. If $F \in \Proj^k_{\Z G}(\Z,\Z)$ is such that $e(F)=-[P_{(G,n)}]$, then
\[
\begin{tikzcd}[row sep=.1cm]
\hProj_{\Z G}^{k+1}(\Z, A;0) \ar[r,"\hspace{2mm} \text{$(- \circ F)^{-1}$}"] & \hProj_{\Z G}^{1}(\Z,A;[P_{(G,n)}]) \ar[r,"\wh e"] & \text{$[P_{(G,n)}]$} \\
\hspace{8mm} E \ar[r,mapsto]  & E' \ar[r,mapsto] & P \oplus \Z G^r \\
\hspace{8mm} E_\theta \ar[r,mapsto] & (m_{\psi_k(\theta)})^*((E')_\theta) \ar[r,mapsto] & ((I,\psi_k(\theta)) \otimes P_\theta) \oplus \Z G^r
\end{tikzcd}
\]
where $P$ is a rank one projective $\Z G$-module and $r \ge 0$.
\end{thm}

\begin{proof}
For the first map, it suffices to check that $(\psi_k(\theta))^*((E')_\theta) \circ F \simeq E_\theta$. Since $E' \circ F = E$, we have $(E')_\theta \circ F_\theta \simeq E_\theta$. By Lemma \ref{lemma:theta-action=pullback}, we have $F_\theta \cong (m_{\psi_n(\theta)})_*(F)$ and so 
\[ E_\theta \simeq (E')_\theta \circ (m_{\psi_n(\theta)})_*(F) = (m_{\psi_n(\theta)})^*((E')_\theta) \circ F.\]
The form for the second map follows directly from Lemma \ref{lemma:image-of-unstable}.
\end{proof}

\begin{thm} \label{thm:action-k|n+2}
Suppose $G$ has $k$-periodic cohomology and $\sigma_k(G) = [P_{(G,n)}]+T_G$ for some $P_{(G,n)} \in P(\Z G)$. If $F \in \Proj^k_{\Z G}(\Z,\Z)$ is such that $e(F)=-[P_{(G,n)}^*]$, then
\[
\begin{tikzcd}[row sep=.1cm, column sep =.4cm] 
\hspace{-1.5mm} \hProj_{\Z G}^{k-1}(\Z, A;0) \ar[r,"\Psi_F"] & \hProj_{\Z G}^{1}(A,\Z;[P_{(G,n)}^*]) \ar[r,"*"] & \hProj^1_{\Z G}(\Z,A^*;[P_{(G,n)}]) \ar[r,"\wh e"] & \text{$[P_{(G,n)}]$} \\
E \ar[r,mapsto]  & E' \ar[r,mapsto] & (E')^* \ar[r,mapsto] & P \oplus \Z G^r \\
E_\theta \ar[r,mapsto] & (m_{\psi_k(\theta)^{\text{-}1}})_*((E')_\theta) \ar[r,mapsto] & (m_{\psi_k(\theta)})^*((E')^*_\theta) \ar[r,mapsto] &  ((I,\psi_k(\theta)) \otimes P_\theta) \oplus \Z G^r
\end{tikzcd}
\]
where $P$ is a rank one projective $\Z G$-module and $r \ge 0$.
\end{thm}

\begin{proof}
For this first map, it suffices to check that $(m_{\psi_k(\theta)^{-1}})_*((E')_\theta) \circ E_\theta \simeq F$. Since $E' \circ E \simeq F$, we have $(E')_\theta \circ E_\theta \simeq F_\theta$. By Lemma \ref{lemma:theta-action=pullback}, we have $F_\theta \cong (m_{\psi_n(\theta)})_*(F)$ and so 
\[ F \simeq (m_{\psi_k(\theta)^{-1}})_*((E')_\theta \circ E_\theta) \simeq (m_{\psi_k(\theta)^{-1}})_*((E')_\theta) \circ E_\theta.\]
For the second map, it is easy to see that pushouts dualise to pullbacks in the other direction, i.e. if $E_0 = (m_{\psi_k(\theta)^{-1}})_*((E')_\theta))$, then $(m_{\psi_k(\theta)^{-1}})_*(E_0^*) \simeq (E')_\theta^*$ and so $E_0^* \simeq (m_{\psi_k(\theta)})^*((E')_\theta^*))$.
The form for the third map follows directly from Lemma \ref{lemma:image-of-unstable}.
\end{proof}

If $G$ has $k$-periodic cohomology and $n=ik$ or $ik-2$ for some $i \ge 1$, then the above shows that the induced action of $\theta \in \Aut(G)$ on $[P_{(G,n)}]$ is given by $P \oplus \Z G^r \mapsto ((I,\psi_{ik}(\theta)) \otimes P_\theta) \oplus \Z G^r$ where $P$ has rank one and $r \ge 0$. 
Furthermore, we have that $\psi_{ik}(\theta) = \psi_k(\theta)^i$ by Lemma \ref{lemma:theta-action-powers}. 

This completes the proof of Theorem \ref{thm:main-topological-II} except for a possible discrepancy in the case where $k=2$ and $i$ is not determined by the fact that $n=ik$ or $ik-2$ (see Remark \ref{remark:k=2}).
However, in this case, $G$ is cyclic and so $(I,r) \cong \Z G$ for all $r \in (\Z/|G|)^\times$ by \cite[Corollary 6.1]{Sw60a}. Hence $(I,\psi_{k}(\theta)^i) \cong \Z G$ is independent of $i$.


\section{Stably free Swan modules and $(G,n)$-complexes} 
\label{section:action-of-Aut(G)}

Before computing the action of $\Aut(G)$ on $[P_{(G,n)}]$, we will pause to consider to consider the role of Swan modules in the classification of $(G,n)$-complexes.
We begin by considering the map
\[ \psi_k : \Aut(G) \to (\Z/|G|)^\times\]
where $G$ has $k$-periodic cohomology.

If $\theta \in \Aut(G)$, then the action $E \mapsto E_\theta$ induces an action of $\Aut(G)$ on $H^k(G;\Z) = \Ext^k_{\Z G}(\Z,\Z)$. This agrees with the usual action coming from the alternate definition of $H^k(-;\Z)$ as a functor on groups \cite[Chapter XII]{CE56}. This implies that $\IM(\psi_k) = \Aut_k(G)$ which is defined in \cite[Section 8]{Dy76}.
We will now give several examples of maps $\psi_k : \Aut(G) \to (\Z /|G|)^\times$.

\subsubsection*{Cyclic}
If $C_n = \langle x \mid x^n=1 \rangle$ is the cyclic group of order $n$, then 
\[\Aut(C_n) = \{ \theta_i: x \mapsto x^i : i \in (\Z/n)^\times\}\] 
and $\psi_2 : \Aut(C_n) \to (\Z/n)^\times$ sends $\theta_i \mapsto i$ by \cite[Proposition 8.1]{Sw60a}. This is surjective and so recovers the classical results $T_{C_n}=1$.

\subsubsection*{Dihedral}
If $D_{4n+2} = \langle x, y \mid x^{2n+1}=y^2=1, yxy^{-1}=x^{-1}\rangle$ is the dihedral group of order $4n+2$, then 
\[\Aut(D_{4n+2}) =\{ \theta_{i,j} : x \mapsto x^i, y \mapsto x^{j}y : i \in (\Z/(2n+1))^\times, j \in \Z/(2n+1)\}\]
and $\psi_4 : \Aut(D_{4n+2}) \to (\Z/(4n+2))^\times$ sends $\theta_{i,j} \mapsto i^2$ by the discussion in \cite[Section 5]{Jo02}. Since $(\Z/(4n+2))^\times = \pm ((\Z/(4n+2))^\times)^2$, this recovers the result $T_{D_{4n+2}}=1$.

\subsubsection*{Quaternionic}
Let $Q_{4n} = \langle x, y \mid x^{n}=y^2, yxy^{-1}=x^{-1}\rangle$ is the quaternion group of order $4n$. For $n=2$, it is shown in \cite[Proposition 8.3]{Sw60a} that $\psi_4 : \Aut(Q_{8}) \to (\Z/8)^\times$ sends $\theta \mapsto 1$ for all $\theta \in \Aut(G)$.
For $n \ge 3$, we have
\[\Aut(Q_{4n}) =\{ \theta_{i,j} : x \mapsto x^i, y \mapsto x^{j}y : i \in (\Z/2n)^\times, j \in \Z/2n\}\]
and $\psi_4 : \Aut(Q_{4n}) \to (\Z/4n)^\times$ sends $\theta_{i,j} \mapsto i^2$ by, for example, \cite[Proposition 1.1]{GG03}.

The following was noted by Davis \cite{Da83} and Dyer \cite[Note (b)]{Dy76}. It would be interesting to know, as was asked by Davis, whether this holds in the case $\sigma_k(G) \ne 0$.

\begin{prop} \label{prop:psi-compososed-with-S}
If $G$ has free period $k$, then $S \circ \psi_k = 0$, i.e. $(I,\psi_k(\theta))$ is stably free for all $\theta \in \Aut(G)$.
\end{prop}

\begin{proof}
Note that Theorems \ref{thm:action-k|n} and \ref{thm:action-k|n+2} each show that $[P] = [(I,\psi_k(\theta)) \otimes P_\theta]$
for all $P \in P(\Z G)$ of rank one such that $\sigma_k(G)=[P]+T_G$. By Lemma \ref{lemma:tensoring-with-swan-modules}, the composition 
\[ \Aut(G) \xrightarrow[]{\psi_k} (\Z /|G|)^\times \xrightarrow[]{S} T_G \le C(\Z G)\]
is given by $S \circ \psi_k: \theta \mapsto [P] - [P_\theta]$ which is well-defined since $\theta$ gives a well-defined action on $C(\Z G)$. 
By Lemma \ref{lemma:theta-basic-facts}, we have that $(\Z G)_\theta \cong \Z G$ and so the composition is trivial in the case where $\sigma_k(G)=0$.
\end{proof}

We say that a finite group $G$ has \textit{weak cancellation} if every stably free Swan module is free. 
The following was asked by Dyer in \cite[p266]{Dy76} and later appeared as Problem A4 in Wall's 1979 Problems List \cite{Wa79b}.

\begin{question} \label{question:swan-module}
Does there exist $G$ with periodic cohomology and $r \in (\Z/|G|)^\times$ such that $(I,r)$ is stably free but not free?	
\end{question}
 
This is equivalent to asking whether every group with periodic cohomology has weak cancellation and is still open, even for arbitrary finite groups. There are two important consequences that a negative answer to Question \ref{question:swan-module} would have.

Firstly, recall the following question from the introduction. 
Note that, if $(I,\psi_k(\theta))$ is free, then the action described in Theorem \ref{thm:main-topological-II} has the simpler form $P \mapsto P_\theta$.
 
\begin{question} \label{question:is-swan-free}
Does there exist $G$ with $k$-periodic cohomology and $\theta \in \Aut(G)$ for which $(I,\psi_k(\theta))$ is not free?	
\end{question}

It follows from Proposition \ref{prop:psi-compososed-with-S} that, if $G$ has free period $k$ and has weak cancellation, then $(I,\psi_k(\theta)) \cong \Z G$ for all $\theta \in \Aut(G)$.
In particular, if Question \ref{question:swan-module} has a negative answer, then the only groups for which the action in Theorem \ref{thm:main-topological-II} might not have the form $P \mapsto P_\theta$ are the groups with $\sigma_k(G) \ne 0$.

Secondly, consider the following: 

\begin{question} \label{question:fork}
Let $n \ge 2$, let $G$ be finite and let $X$, $Y$ be finite $(G,n)$-complexes with $\chi(X) = \chi(Y)$. Then $X \vee rS^n \simeq Y \vee rS^n$ for some $r$. Does $r=1$ always work?
\end{question}

This is equivalent to asking whether $\HT(G,n)$ is a fork when $G$ is finite. The case where $n$ is even was proven by Browning \cite{Br79a}, and also follows by combining Corollary \ref{cor:fork} and Propostion \ref{prop:pht-to-alg}.
When $n$ is odd, this is known to hold provided $G$ does not have $k$-periodic cohomology for any $k \mid n+1$. If $G$ has $k$-periodic cohomology for $k \mid n+1$, then this holds provided $G$ has weak cancellation (see \cite[p276-277]{Dy76}). In particular, if Question \ref{question:swan-module} has a negative answer, then Question \ref{question:fork} has an affirmative answer. Note that the corresponding question for infinite groups is also still open (see \cite[Problem 2]{Ni21a}).


\section{Milnor squares and the classification of projective modules}
\label{section:Milnor-squares}

Given the observations in the previous section, the primary obstacle to computing sufficiently interesting examples of $\HT(G,n)$ and $\PHT(G,n)$ for our groups is the classification of projective $\Z G$-modules. 

One method to classify projective $R$-modules over a ring $R$ is to relate this to the classification of projective modules over simpler rings using Milnor squares. In this section, we will present a refinement of the basic theory of Milnor squares which will also allow us to determine how a ring automorphism $\alpha \in \Aut(R)$ acts on the class of projective $R$-modules. We will then apply these methods in Section \ref{section:examples}. 

Suppose $R$ and $S$ are rings and $f: R \to S$ is a ring homomorphism. We can use this to turn $S$ into an $(S,R)$-bimodule, with right-multiplication by $r \in R$ given by $x \cdot r = x f(r)$ for any $x \in S$. If $M$ is an $R$-module, we can define the \textit{extension of scalars} of $M$ by $f$ as the tensor product
\[ f_\#(M) = S \otimes_{R} M \]
since $S$ as a right $R$-module and $M$ as a left $R$-module, and we consider this as a left $S$-module where left-multiplication by $s \in S$ is given by $s \cdot (x \otimes m) = (sx)\otimes m$ for any $x \in S$ and $m \in M$. 
This comes equipped with maps of abelian groups
\[ f_*: M \to f_\#(M)\]
sending $m \mapsto 1 \otimes m$, and defines a covariant functor from $R$-modules to $S$-modules \cite[p227]{CR81}. It has the following basic properties which follow from the standard properties of tensor products such as associativity \cite[p145]{Ma63}.

\begin{lemma} \label{lemma:basic-properties:ext-of-scalars}
Let $f :R \to S$ and $g: S \to T$ be ring homomorphisms and let $M$ and $N$ be $R$-modules. Then
\begin{enumerate}[\normalfont(i)]
\item $f_\#(M \oplus N) \cong f_\#(M) \oplus f_\#(N)$
\item $f_\#(R) \cong S$
\item $(g \circ f)_\#(M) \cong (g_\# \circ f_\#)(M)$.
\end{enumerate}
\end{lemma}

If $P(R)$ denotes the set of isomorphism classes of projective $R$-modules, then the first two properties show that $f_\#$ induces a map
$f_\# : P(R) \to P(S)$
which restricts to each stable class. 

Recall that, if $R$, $R_1$, $R_2$ and $R_0$ are rings, then a pullback diagram
\[
\mathcal{R} = 
\begin{tikzcd}
  R \arrow[r, "i_2"] \arrow[d, "i_1"] & R_2 \arrow[d,"j_2"] \\
  R_1 \arrow[r,"j_1"] & R_0
\end{tikzcd}
\]
is a Milnor square if either $j_1$ or $j_2$ are surjective. If $P_1 \in P(R_1)$, $P_2 \in P(R_2)$ are such that there is a $R_0$-module isomorphism $h: (j_1)_\#(P_1) \to (j_2)_\#(P_2)$, then define
	\[ M(P_1,P_2,h) = \{ (x,y) \in P_1 \times P_2 : h((j_1)_*(x))=(j_2)_*(y)\} \le P_1 \times P_2, \]
	which is an $R$-module where multiplication by $r \in R$ is given by $r \cdot (x,y) = ((i_1)_*(r)x,(i_2)_*(r)y)$. It was shown by Milnor that $M(P_1,P_2,h)$ is projective \cite[Theorem 2.1]{Mi71}. 
	Let $\Aut_{R}(P)$ denote the set of $R$-module automorphisms of an $R$-module $P$. The main result on Milnor squares is as follows. This is a consequence of the results in \cite[Section 2]{Mi71} and the precise statement can be found in \cite[Proposition 4.1]{Sw80}.

\begin{thm} \label{thm:milnor-square}
Suppose $\mathcal{R}$ is a Milnor square and $P_i \in P(R_i)$ for $i=0,1,2$ are such that $P_0 \cong (j_1)_\#(P_1) \cong (j_2)_\#(P_2)$ as $R_0$-modules. Then there is a one-to-one correspondence
\[ \Aut_{R_1}(P_1) \backslash \Aut_{R_0}(P_0) \slash \Aut_{R_2}(P_2) \leftrightarrow \{ P \in P(R) : (i_1)_\#(P) \cong P_1, (i_2)_\#(P) \cong P_2 \} \]
given by sending a coset $[h]$ to $M(P_1,P_2,h)$ for any representative $h$.
\end{thm}

Now suppose $\alpha \in \Aut(R)$. If $M$ is an $R$-module, define $M_{\alpha}$ as the $R$-module whose abelian group is that of $M$ but whose $R$-action is given by $r \cdot m = \alpha(r) m$ for $r \in R$ and $m \in M$. 
For example, if $R = \Z G$, then $\theta \in \Aut(G)$ induces a map $\theta \in \Aut(\Z G)$ and $M_\theta$ coincides with the definition given earlier. 

This is a special case of restriction of scalars, but can also be viewed as a part of extension of scalars as follows.

\begin{lemma} \label{lemma:EOS-for-endomorphisms}
Let $R$ be a ring and let $\alpha \in \Aut(R)$. If $M$ is an $R$-module, then there is an isomorphism of $R$-modules
\[ \psi: M_\alpha \to (\alpha^{-1})_\#(M)\]
given by sending $m \mapsto 1 \otimes m$.
\end{lemma}

From this, it is clear that this action has basic properties which are analogous to Lemma \ref{lemma:theta-basic-facts}.
The following is then immediate by combining Lemmas \ref{lemma:basic-properties:ext-of-scalars} and \ref{lemma:EOS-for-endomorphisms}.

\begin{corollary} \label{cor:ext-of-scalars-alpha-action}
Suppose $f: R \to S$ is a ring homomorphism and $\alpha \in \Aut(R)$, $\beta \in \Aut(S)$ are such that $f \circ \alpha = \beta \circ f$. If $M$ is an $R$-module, then
\[ f_\#(M_{\alpha}) \cong f_\#(M)_{\beta}.\]
\end{corollary}

We can turn the set of Milnor squares into a category with morphisms defined as follows. If $\mathcal{R}$, $\mathcal{R}'$ are Milnor squares, then a morphism is a quadruple
\[ \hat{\alpha} =(\alpha,\alpha_1,\alpha_2,\alpha_0): \mathcal{R} \to \mathcal{R}' \]
where $\alpha: R \to R'$ and $\alpha_i: R_i \to R_i'$ such that there is a commutative diagram as follows
\[
    \begin{tikzcd}[row sep=1.5em, column sep = 1.5em]
    R \arrow[rr] \arrow[dr,"\alpha"] \arrow[dd,swap] &&
    R_2 \arrow[dd] \arrow[dr,"\alpha_2"] \\
    & R \arrow[rr] \arrow[dd] &&
    R_2 \arrow[dd] \\
    R_1 \arrow[rr,] \arrow[dr, "\alpha_1"] && R_0 \arrow[dr, "\alpha_0"] \\
    & R_1 \arrow[rr]&& R_0
    \end{tikzcd}
\]

Let $\Aut(\mathcal{R})$ denote the set of automorphisms of a Milnor square $\mathcal{R}$, i.e. the set of isomorphisms $\hat{\alpha} : \mathcal{R} \to \mathcal{R}$.

\begin{lemma} \label{lemma:milnor-action}
Let $\mathcal{R}$ be a Milnor square, let $P_1 \in P(R_1)$, $P_2 \in P(R_2)$ be such that there is an $R_0$-module isomorphism $h: (j_1)_\#(P_1) \to (j_2)_\#(P_2)$. If $\hat{\alpha} =(\alpha,\alpha_1,\alpha_2,\alpha_0) \in \Aut(\mathcal{R})$, then
\[M(P_1,P_2,h)_\alpha \cong M((P_1)_{\alpha_1},(P_2)_{\alpha_2},h) \]	
where, on the right, we view $h$ as a map $h: (j_1)_\#(P_1)_{\alpha_0} \to (j_2)_\#(P_2)_{\alpha_0}$.
\end{lemma}

\begin{proof}
Let $P = M(P_1,P_2,h)$ so that, by Theorem \ref{thm:milnor-square}, we have that $(i_1)_\#(P) \cong P_1$ and $(i_2)_\#(P) \cong P_2$. It is easy to see directly that the natural map
\[ M((i_1)_\#(P), (i_2)_\#(P), h) \to  M((i_1)_\#(P_\alpha), (i_2)_\#(P_\alpha), h)\]
is an isomorphism. We are then done by applying Corollary \ref{cor:ext-of-scalars-alpha-action}.
\end{proof}

This has the following simplification when $P_1$ and $P_2$ are free of rank one. Here we will use the identification $\Aut_{R_0}(R_0) \cong R_0^\times$ which sends $h:R_0 \to R_0$ to $h(1) \in R_0^\times$.

\begin{lemma} \label{lemma:milnor-free-case}
Let $\mathcal{R}$ be a Milnor square and let $u \in R_0^\times$. If $\hat{\alpha} =(\alpha,\alpha_1,\alpha_2,\alpha_0) \in \Aut(\mathcal{R})$, then
\[M(R_1,R_2,u)_\alpha \cong M(R_1,R_2,\alpha_0^{-1}(u)).\]	
\end{lemma}

\begin{proof}
Fix identifications $\psi_i : (j_i)_\#(R_i) \to R_0$ and let $h : (j_1)_\#(R_1) \to (j_1)_\#(R_1)$ be such that $(\psi_2 \circ h \circ \psi_1^{-1})(1)=u \in R_0^\times$. By Lemma \ref{lemma:milnor-action}, we have that 
\[ M(R_1,R_2,h)_\alpha \cong M((R_1)_{\alpha_1},(R_2)_{\alpha_2},h)\] 
where $h : ((j_1)_\#(R_1))_{\alpha_0} \to ((j_1)_\#(R_1))_{\alpha_0}$ coincides with $h$ as a map of abelian groups.
For $i=0,1,2$, let $c_i : R_i \to (R_i)_{\alpha_i}$ be the isomorphism which sends $1 \mapsto 1$. Then it is easy to see that the following diagram commutes for $i=1,2$, where $f: (j_i)_\#((R_i)_{\alpha_i}) \to ((j_i)_\#(R_i))_{\alpha_0}$ is the isomorphism coming from Corollary \ref{cor:ext-of-scalars-alpha-action}
\[
\begin{tikzcd}
(j_i)_\#(R_i) \ar[r,"1\otimes c_i"] \ar[d,"\psi_i"] & ((j_i)_\#((R_i)_{\alpha_i}) \ar[r,"f"] & ((j_i)_\#(R_i))_{\alpha_0} \ar[d,"\psi_i"] \\
R_0 \ar[rr,"c_0"] && (R_0)_{\alpha_0}
\end{tikzcd}
\]
Using the isomorphisms $c_i$ for $i=1,2$, we get that
\[ M((R_1)_{\alpha_1},(R_2)_{\alpha_2},h) \cong M(R_1,R_2,h_0)\]
where $h_0 : (j_1)_\#(R_1) \to (j_2)_\#(R_2)$ induces $h : ((j_1)_\#(R_1))_{\alpha_0} \to ((j_1)_\#(R_1))_{\alpha_0}$ via $f \circ (1 \otimes c_i)$. Let $u_0 = (\psi_2 \circ h_0 \circ \psi_1^{-1})(1) \in R_0^\times$. Then, since the above diagram commutes, we get the following commutative diagram
\[ 
\begin{tikzcd}
	R_0 \ar[r,"\psi_2 \circ h \circ \psi_1^{-1}"] \ar[d,"c_0"] & R_0 \ar[d,"c_0"]\\
	(R_0)_{\alpha_0} \ar[r,"\psi_2 \circ h_0 \circ \psi_1^{-1}"] & (R_0)_{\alpha_0}
\end{tikzcd}
\quad
\begin{tikzcd}
1 \ar[r,mapsto] \ar[d,mapsto] & u_0 \ar[d,mapsto] \\
1 \ar[r,mapsto] & \alpha_0(u_0)
\end{tikzcd}
\]
This implies that $u=\alpha_0(u_0)$ and so $u_0 = \alpha_0^{-1}(u)$, as required.
\end{proof}

If $\mathcal{R}$ is a Milnor square, we say that $\alpha \in \Aut(R)$ \textit{extends across $\mathcal{R}$} if there exists $\hat{\alpha} = (\alpha, \alpha_1,\alpha_2,\alpha_0) \in \Aut(\mathcal{R})$. 
The following gives conditions under which this induced map is unique.

\begin{lemma} \label{lemma:aut-unique}
Let $\mathcal{R}$ be a pullback square with all maps surjective. If $\alpha \in \Aut(R)$ extends across $\mathcal{R}$, then it does so uniquely. That is, there exist unique maps $\alpha_1$, $\alpha_2$, $\alpha_0$ for which $\hat{\alpha} = (\alpha, \alpha_1,\alpha_2,\alpha_0) \in \Aut(\mathcal{R})$.
\end{lemma}

\begin{proof}
This follows from the simple observation that, if $f: R \twoheadrightarrow S$ is a surjective ring homomorphism and $\alpha: R \to R$ and $\beta_1, \beta_2 : S \to S$ are ring homomorphisms such that $f \circ \alpha = \beta_i \circ f$ for $i=1,2$, then $\beta_1 = \beta_2$. To see this, note that the conditions imply that $(\beta_1-\beta_2) \circ f = 0$ and so $\beta_1 = \beta_2$ on $\IM(f)$. Since $f$ is surjective, $\IM(f) = S$ and so $\beta_1=\beta_2$.
\end{proof}

We conclude this section with the following result which is a consequence of Theorem \ref{thm:milnor-square}, Lemma \ref{lemma:milnor-free-case} and Lemma \ref{lemma:aut-unique}.

\begin{prop} \label{prop:milnor-square-action}
Let $\mathcal{R}$ be a pullback square with all maps surjective and such that every $\alpha \in \Aut(R)$ extends across $\mathcal{R}$. Then there is a one-to-one correspondence
\[ R_1^\times \backslash (R_0^\times /\Aut(R)) \slash R_2^\times \leftrightarrow \{ P \in P(R) : (i_1)_\#(P) \cong R_1, (i_2)_\#(P) \cong R_2 \}/\Aut(R) \]
where $\alpha \in \Aut(R)$ acts on $R_0^\times$ by sending $r \mapsto \alpha_0^{-1}(r)$ for $r \in R_0^\times$ and where $\alpha_0 \in \Aut(R_0)$ is the unique automorphism such that $\hat{\alpha} = (\alpha, \alpha_1,\alpha_2,\alpha_0) \in \Aut(\mathcal{R})$.
\end{prop}

\section{Example: Quaternion Groups} 
\label{section:examples}

The aim of this section will be to illustrate how Theorems \ref{thm:main-topological-I} and \ref{thm:main-topological-II} can be combined with the known techniques to classify projective $\Z G$-modules to obtain a detailed classification of finite $(G,n)$-complexes up to homotopy equivalence.

For $k \ge 2$, recall that the quaternion group of order $4k$ has presentation:
\[ Q_{4k} = \langle x, y \mid x^{k}=y^2, yxy^{-1}=x^{-1}\rangle. \]
It is a finite $3$-manifold group and so has free period $4$. For $n \ge 2$ even, Theorem \ref{thm:main-topological-I} and Proposition \ref{prop:pht-to-alg} imply that $\PHT(Q_{4k},n) \cong [\Z Q_{4k}] = \bigcup_{r \ge 1} \SF_r(\Z Q_{4k})$ where $\SF_r(\Z Q_{4k})$ is the set of stably free $\Z Q_{4k}$-modules of rank $r \ge 1$. 

Since stably free $\Z G$-modules of rank $\ge 2$ are free for $G$ finite \cite{Sw60b} (or since $\PHT(G,n)$ is a fork by Corollary \ref{cor:fork}), it remains to compute $\SF_1(\Z Q_{4k})$. This was completed by Swan for $k \le 9$ \cite[Theorem III]{Sw83}. For $k \le 7$, he showed that $|\SF_1(\Z Q_{4k})| = 1$ for $2 \le k \le 5$, $|\SF_1(\Z Q_{24})|=3$ and $|\SF_1(\Z Q_{28})|=2$.
It also follows from his classification that $\Z Q_{4k}$ has weak cancellation in all these cases and so the action of $\theta \in \Aut(Q_{4k})$ on $[\Z Q_{4k}]$ sends $P \mapsto P_\theta$ (see Section \ref{section:action-of-Aut(G)}).

In the case $Q_{28}$, the action of $\Aut(Q_{28})$ on $[\Z Q_{28}]$ is trivial since $(\Z Q_{28})_\theta \cong \Z Q_{28}$ for all $\theta \in \Aut(Q_{28})$ and so this must also hold for the non-free stably free module also. The main result of this section will be to compute the action in the case $Q_{24}$.

\begin{thm} \label{thm:main-example}
$\Aut(Q_{24})$ acts non-trivially on $[\Z Q_{24}]$. More specifically, we have $|\SF_1(\Z Q_{24})| = 3$ and $|\SF_1(\Z Q_{24})/\Aut(Q_{24})| = 2$. 
\end{thm}

All this can be summarised in the following table, which gives the structure of $\PHT(G,n)$ and $\HT(G,n)$ when $n \ne 2$ is even. These graded trees are both forks by Corollary \ref{cor:fork} and each dot represents a finite $(G,n)$-complex at the minimal level.

\begin{figure}[h]
\normalfont
\begin{tabular}{|c|c|c|c|c|c|c|}
 \hline 
$G$ & $Q_8$ & $Q_{12}$ & $Q_{16}$ & $Q_{20}$ & $Q_{24}$ & $Q_{28}$ \\
 \hline
  $\PHT(G,n)$ & $\bullet$ & $\bullet$ & $\bullet$ & $\bullet$ & $\bullet$ $\bullet$ $\bullet$ & $\bullet$ $\bullet$ \\
 \hline
 $\HT(G,n)$ & $\bullet$ & $\bullet$ & $\bullet$ & $\bullet$ & $\bullet$ $\bullet$ & $\bullet$ $\bullet$ \\
  \hline
\end{tabular}
\caption{Minimal complexes for any $n$ even with $n \ne 2$} \vspace{-2mm}
\label{figure:quaternion-table}
\end{figure}

\begin{remark}
This also holds in the case $n=2$ provided $G$ has the D2 property. This holds trivially in the cases $Q_8, Q_{12}, Q_{16}, Q_{20}$ and is otherwise only known to be true in the case $Q_{28}$ by \cite[Theorem 7.7]{Ni19} using the presentation of Mannan-Popiel \cite{MP19}.
\end{remark}

We will now proceed to the proof of Theorem \ref{thm:main-example}. First let $x$ and $y$ be generators for $Q_{24}$ in the presentation given above. Let $\Lambda = \Z Q_{24}/(x^6+1)$ and note that the quotient
map $f : \Z Q_{24} \twoheadrightarrow \Lambda$ induces a map
\[ f_\# : \SF_1(\Z Q_{24}) \to \SF_1(\Lambda)\]
by Lemma \ref{lemma:basic-properties:ext-of-scalars}. This is a bijection by the proof of \cite[Theorem 11.14]{Sw83}.

Now note that the factorisation $x^6+1 = (x^2+1)(x^4-x^2+1)$ implies that the ideals $I = (x^2+1)$ and $J = (x^4-x^2+1)$ have $I \cap J = (x^6+1)$ and $I + J = (3,x^2+1)$.
It follows from \cite[Example 42.3]{CR87} that we have a pullback diagram
\[
\begin{tikzcd}
	\Lambda \ar[r] \ar[d] & \Z Q_{24}/(x^4-x^2+1) \ar[d]\\
	\Z Q_{24}/(x^2+1) \ar[r] & \F_3 Q_{24}/(x^2+1)
\end{tikzcd}
\]
which is a Milnor square since all maps are surjective.

For a field $\F$, let $\H_{\F} = \F[i,j]$ denote the quaternions over $\F$ and we define $\H_{\Z}=\Z[i,j]$ and $\Z[\zeta_{12},j]$ to be subrings of $\H_{\R}$, where $\zeta_{12} = e^{\frac{2\pi i}{12}}$ is the 12th root of unity in the $i$ direction. The following is straightforward to check that there are isomorphisms of rings
\begin{align*} \phi_1: \H_{\Z} &\to \Z Q_{24} / (x^2+1), & \phi_2 : \Z[\zeta_{12},j] &\to \Z Q_{24} / (x^4-x^2+1) \\
i \mapsto &\, x, \, j \mapsto y & \zeta_{12} \mapsto &\, x, \, j \mapsto y.\end{align*}

Using this, we can rewrite the Milnor square above as follows
\[
\mathcal{R} = 
\begin{tikzcd}
	\Lambda \ar[r,"i_2"] \ar[d,"i_1"] & \Z[\zeta_{12},j] \ar[d,"j_2"]\\
	\H_{\Z} \ar[r,"j_1"] & \H_{\F_3}
\end{tikzcd}
\quad
\begin{tikzcd}
	x,y \ar[r,mapsto] \ar[d,mapsto] & \zeta_{12},j \ar[d,mapsto]\\
	i,j \ar[r,mapsto] & i,j
\end{tikzcd}
\]
Now note that, by \cite[Lemma 8.14]{Sw83}, the induced map $(i_2)_*: C(\Lambda) \to C(\Z[\zeta_{12},j])$ is an isomorphism. It also follows from \cite[p84]{Sw83} that the rings $\H_{\Z}$ and $\Z[\zeta_{12},j]$ both have stably free cancellation, i.e. that every stably free module is free. It follows easily that
\[ \SF_1(\Lambda) = \{P \in P(\Lambda) : (i_1)_\#(P) \cong \H_{\Z}, (i_2)_\#(P) \cong \Z[\zeta_{12},j]\}. \]
In particular, by combining with Theorem \ref{thm:milnor-square}, we get that there is a bijection
\[ \SF_1(\Lambda) \leftrightarrow \H_{\Z}^\times \backslash \H_{\F_3}^\times \slash \Z[\zeta_{12},j]^\times.\]

\begin{lemma} \label{lemma:double-coset-calc}
$\H_{\Z}^\times \backslash \H_{\F_3}^\times \slash \Z[\zeta_{12},j]^\times = \{[1],[1+j],[1+k] \}$.
\end{lemma}

\begin{proof}
If $N: \H_{\F_3} \to \F_3$ is the norm, then $\H_{\F_3}^\times = N^{-1}(\pm 1)$. Now note that $\H_{\Z}^\times = \{\pm 1, \pm i, \pm j, \pm k\}$, and it is easy to check that
\[ \H_{\Z}^\times \backslash \H_{\F_3}^\times = \{ [1],[1+i],[1+j],[1+k],[1+i+j+k],[1-i-j-k]\}.\]
By \cite[Lemma 7.5 (b)]{MOV83}, we have that $\Z[\zeta_{12},j]^\times = \Z[\zeta_{12}]^\times \cdot \langle j \rangle$ and so it remains to determine
\[ \IM(\Z[\zeta_{12},j]^\times \to \H_{\Z}^\times \backslash \H_{\F_3}^\times)= \IM(\Z[\zeta_{12}]^\times \to \H_{\Z}^\times \backslash \H_{\F_3}^\times) \subseteq \{[1],[1+i]\},\]
where the last inclusion follows since $\zeta_{12} \mapsto i$ and $\H_{\Z}^\times \backslash \langle 1, i \rangle = \{[1],[1+i]\}$.

Consider the $n$th cyclotomic polynomial
\[ \Phi_n(x) = \prod_{k \in \Z_{n}^\times}(x-\zeta_{n}^k).\]
It is well-known, and can be shown using M\"{o}bius inversion, that $\Phi_{n}(1)=1$ if $n$ is not a prime power. In particular, $\Phi_{12}(1)=1$ and this implies that $1 - \zeta_{12} \in \Z[\zeta_{12}]^\times$. Hence $[1+i] = [1-i] \in \IM(\Z[\zeta_{12}]^\times \to \H_{\Z}^\times \backslash \H_{\F_3}^\times)$. The result then follows since
\[ j(1+i+j+k)(1+i)=1+k, \quad -j(1-i-j-k)(1+i)=1+j\]
implies that $[1+j] = [1-i-j-k]$, $[1+k] = [1+i+j+k]$ in $\H_{\Z}^\times \backslash \H_{\F_3}^\times \slash \Z[\zeta_{12},j]^\times$.
\end{proof}

This implies that $|\SF_1(\Z Q_{24})|=3$, which recovers the result of Swan. In order to determine the action of $\Aut(Q_{24})$ on $\SF_1(\Z Q_{24})$, first recall from Section \ref{section:action-of-Aut(G)} that 
\[ \Aut(Q_{24}) =\{ \theta_{a,b} : x \mapsto x^a, y \mapsto x^{b}y \mid a \in (\Z/12)^\times, b \in \Z/12\}.\]
If $\mathcal{R}$ denote the Milnor square defined above, then the following is easy to check.

\begin{lemma} \label{lemma:induced-action}
If $a \in (\Z/12)^\times$, $b \in \Z/12$, then $\theta_{a,b} \in \Aut(Q_{24})$ extends to a Milnor square automorphism
\[ \hat{\theta}_{a,b} = (\theta_{a,b}',\theta_{a,b}^1,\theta_{a,b}^2,\bar{\theta}_{a,b}) \in \Aut(\mathcal{R})\]
where, for $a=2a_0+1$, the maps are defined as follows:
\begin{enumerate}[\normalfont (i)]
\item $\theta_{a,b}' \in \Aut(\Z Q_{24}/(x^6+1))$ is given by $x \mapsto x^a$, $y \mapsto x^{b}y$.

	\item
	$\theta_{a,b}^1 \in \Aut(\H_{\Z})$ and $\bar{\theta}_{a,b} \in \Aut(\H_{\F_3})$ are each given by
	\[ i \mapsto i^a = (-1)^{a_0} i, \quad  j \mapsto j^b =
\begin{cases}
	(-1)^{b_0}j, & \text{if $b=2b_0+1$}\\
		(-1)^{b_0}k, & \text{if $b=2b_0$.}
\end{cases}
\]
	\item
	$\theta_{a,b}^2 \in \Aut(\Z[\zeta_{12},j])$ is given by $\zeta_{12} \mapsto \zeta_{12}^a$, $j \mapsto \zeta_{12}^b j$.
\end{enumerate}
\end{lemma}

Since $\mathcal{R}$ is a pullback square with all maps surjective, we can now apply Proposition \ref{prop:milnor-square-action}.
By combining with Lemma \ref{lemma:double-coset-calc}, this implies that there is a bijection
\[ \SF_1(\Z Q_{24})/\Aut(Q_{24}) \leftrightarrow \{[1],[1+j],[1+k] \}/\Aut(Q_{24})\]
where $\theta_{a,b} \in \Aut(Q_{24})$ acts on the double cosets via the action described in Lemma \ref{lemma:induced-action}. In particular
\[ \bar{\theta}_{a,b}([1+j]) = 
\begin{cases}
 	[1+(-1)^{b_0}j] = [1+j], & \text{if $b=2b_0+1$}\\
	[1+(-1)^{b_0}k]=[1+k], & \text{if $b=2b_0$}
\end{cases}
\]
and so $\bar{\theta}_{a,b}$ acts non-trivially when $b$ is even. Hence $|\SF_1(\Z Q_{24})/\Aut(Q_{24})|=2$. This completes the proof of Theorem \ref{thm:main-example}.

\bibliography{biblio.bib}
\bibliographystyle{amsplain}

\end{document}